\documentclass[oneside,english,a4paper]{amsart}
\usepackage[T1]{fontenc}
\usepackage[latin9]{inputenc}
\usepackage{amstext}
\usepackage{amsthm}
\usepackage{amssymb}

\makeatletter
\numberwithin{equation}{section}
\numberwithin{figure}{section}
\theoremstyle{plain}
\newtheorem{thm}{\protect\theoremname}
  \theoremstyle{remark}
  \newtheorem*{nota}{Notation}
  \newtheorem{rem}[thm]{\protect\remarkname}
  \theoremstyle{definition}
  \newtheorem{defn}[thm]{\protect\definitionname}
  \theoremstyle{plain}
  \newtheorem{prop}[thm]{\protect\propositionname}
  \theoremstyle{plain}
  \newtheorem{lem}[thm]{\protect\lemmaname}

\makeatother

\usepackage{babel}
  \providecommand{\definitionname}{Definition}
  \providecommand{\lemmaname}{Lemma}
  \providecommand{\propositionname}{Proposition}
  \providecommand{\remarkname}{Remark}
\providecommand{\theoremname}{Theorem}

\begin{document}

\title[Compactness for linearly perturbed Yamabe problem]{Compactness results for linearly perturbed Yamabe problem on manifolds
with boundary}

\author{Marco Ghimenti}
\address{M. Ghimenti, \newline Dipartimento di Matematica Universit\`a di Pisa
Largo B. Pontecorvo 5, 56126 Pisa, Italy}
\email{marco.ghimenti@unipi.it}

\author{Anna Maria Micheletti}
\address{A. M. Micheletti, \newline Dipartimento di Matematica Universit\`a di Pisa
Largo B. Pontecorvo 5, 56126 Pisa, Italy}
\email{a.micheletti@dma.unipi.it.}

\begin{abstract}
Let $(M,g)$ a compact Riemannian $n$-dimensional manifold. It is
well know that, under certain hypothesis, in the conformal class of
$g$ there are scalar-flat metrics that have $\partial M$ as a constant
mean curvature hypersurface. Also, under certain hypothesis, it is
known that these metrics are a compact set. In this paper we prove
that, both in the case of umbilic and non-umbilic boundary, if we linearly perturb the mean curvature
term $h_{g}$ with a negative smooth function $\alpha,$ the set of
solutions of Yamabe problem is still a compact set.
\end{abstract}


\maketitle

\section{Introduction}

Let $(M,g)$, a smooth, compact Riemannian manifold of dimension $n\ge3$
with boundary. In \cite{Es} Escobar asked it there exists a conformal
metric $\tilde{g}=u^{\frac{4}{n-2}}g$ for which $M$ has zero scalar
curvature and constant boundary mean curvature. 

This problem can be understood as a generalization of the Riemann
mapping theorem and it is equivalent to finding a positive solution
to the following nonlinear boundary value problem

\begin{equation}
\left\{ \begin{array}{cc}
L_{g}u=0 & \text{ in }M\\
B_{g}u+(n-2)u^{\frac{n}{n-2}}=0 & \text{ on }\partial M
\end{array}\right..\label{eq:Pconf}
\end{equation}
Where $L_{g}=\Delta_{g}-\frac{n-2}{4(n-1)}R_{g}$ and $B_{g}=-\frac{\partial}{\partial\nu}-\frac{n-2}{2}h_{g}$
are respectively the conformal Laplacian and the conformal boundary
operator, $R_{g}$ is the scalar curvature of the manifold, $h_{g}$
is the mean curvature of the $\partial M$ and $\nu$ is the outer
normal with respect to $\partial M$ .

The existence of solutions of (\ref{eq:Pconf}) was established by
the works of Escobar \cite{Es}, Marquez \cite{M1}, Almaraz \cite{Al},
Chen \cite{ch}, Mayer and Ndiaye \cite{MN}.

Solutions of (\ref{eq:Pconf}) are the critical points of the functional
quotient
\[
Q(u):=\inf_{u\in H^{1}\smallsetminus0}\frac{\int\limits _{M}\left(|\nabla u|^{2}+\frac{n-2}{4(n-1)}R_{g}u^{2}\right)dv_{g}+\int\limits _{\partial M}\frac{n-2}{2}h_{g}u^{2}d\sigma_{g}}{\left(\int\limits _{\partial M}|u|^{\frac{2(n-1)}{n-2}}d\sigma_{g}\right)^{\frac{n-2}{n-1}}}.
\]

In \cite{Es} Escobar introduced, in analogy of the classical Yamabe
problem
\[
Q(M,\partial M):=\inf\left\{ Q(u)\ :\ u\in H^{1}(M),\ u\neq0\text{ on }\partial M\right\} .
\]
 Concerning the compactness of the full set of positive solutions
of (\ref{eq:Pconf}), the only interesting case occurs when $Q>0$.
Indeed, when $Q<0$ the solution $u$ is unique while when $Q=0$
the solution is unique up to positive multiplicative constants. 

First compactness results have be proven by Felli and Ould Ahmedou
\cite{FA} for any $n\ge3$, in the case of locally conformally flat
manifolds and by Almaraz in \cite{Al} for $n\ge7$, in the case of
manifolds with nonumbilic boundary.

We recall that the boundary of $M$ is respectively called umbilic
(nonumbilic) if the trace-free second fundamental form of $\partial M$
is zero (different from zero) everywhere.

If either $n>8$ and the Weyl tensor of $M$ never vanishes on $\partial M$
or $n=8$ and the Weyl tensor of $\partial M$ never vanishes on $\partial M$,
the compactness is still true for manifolds with umbilic boundary \cite{GM}.

Very recently the compactness was showed for manifold of dimension
$n=3$ \cite{AQW}, $n=4$ \cite{KMW} and -when the boundary is nonumbilic-
$n=5,6$ \cite{KMW}.

An interesting point is the stability problem that is if the compactness
is preserved under small perturbations of the equation (\ref{eq:Pconf}). 

In particular we consider the linear perturbation problem
\begin{equation}
\left\{ \begin{array}{cc}
L_{g}u=0 & \text{ in }M\\
\frac{\partial u}{\partial\nu}+\frac{n-2}{2}h_{g}u+\varepsilon\alpha u=(n-2)u^{\frac{n}{n-2}} & \text{ on }\partial M
\end{array}\right.\label{eq:Prob-2}
\end{equation}
where $\varepsilon$ is a small positive parameter and $\alpha:M\rightarrow\mathbb{R}$
is a smooth function.

We can prove that the sign of the function $\alpha$ on $\partial M$
has an effect on compactness and non compactness of solutions of (\ref{eq:Prob-2}):
in \cite{GMP18} we proved the existence of blowing up solution of
(\ref{eq:Prob-2}) when $\alpha>0$ in the case of $\partial M$ non
umbilic and $n\ge7$ and in \cite{GMP} we proved an analogous result
in the case of $n\ge11$ and the Weyl tensor not vanishing on $\partial M$.

In the following we show that when $\alpha$ is negative everywhere
on $\partial M$ there are no blowing up solutions for $\varepsilon\rightarrow0$,
i.e. compactness holds. This is analogous of what happens when perturbing
the Scalar curvature term in the classical Yamabe problem (see \cite{dru,DH}
and the references therein)

Our main results are
\begin{thm}
\label{thm:main}Let $(M,g)$ a smooth, $n$-dimensional Riemannian
manifold of positive type not conformally equivalent to the standard
ball with regular umbilic boundary $\partial M$. 

Let $\alpha:M\rightarrow\mathbb{R}$ such that $\alpha<0$ on $\partial M$.
Suppose that $n>8$ and that the Weyl tensor $W_{g}$ is not vanishing
on $\partial M$ or suppose that $n=8$ and that the Weyl tensor referred
to the boundary $\bar{W}_{g}$ is not vanishing on $\partial M$.
Then, given $\bar{\varepsilon}>0$ there exists a positive constant
$C$ such that for any $\varepsilon\in(0,\bar{\varepsilon})$ and
for any $u>0$ solution of (\ref{eq:Prob-2}) it holds 
\[
C^{-1}\le u\le C\text{ and }\|u\|_{C^{2,\eta}(M)}\le C
\]
for some $0<\eta<1$. The constant $C$ does not depend on $u,\varepsilon$. 
\end{thm}
\begin{thm}
\label{thm:main-1}Let $(M,g)$ a smooth, $n$-dimensional Riemannian
manifold of positive type with non umbilic boundary $\partial M$,
with $n\ge7$. 

Let $\alpha:M\rightarrow\mathbb{R}$ such that $\alpha<0$ on $\partial M$.
Then, given $\bar{\varepsilon}>0$ there exists a positive constant
$C$ such that for any $\varepsilon\in(0,\bar{\varepsilon})$ and
for any $u>0$ solution of (\ref{eq:Prob-2}) it holds 
\[
C^{-1}\le u\le C\text{ and }\|u\|_{C^{2,\eta}(M)}\le C
\]
for some $0<\eta<1$. The constant $C$ does not depend on $u,\varepsilon$. 
\end{thm}

\subsection{Structure of the paper}

We will give the proof of Theorem \ref{thm:main} in full detail in
Section \ref{sec:Main-Proof}, while in Section \ref{sec:almaraz}
we will give only the main ingredients to prove Theorem \ref{thm:main-1}
following the same strategy of Thm \ref{thm:main}. In Section \ref{sec:Pohozaev}
we recall a version of Pohozaev identity for Problem (\ref{eq:Prob-2}).
In Section \ref{sec:Expansion} we choose a suitable metric conform
to the given metric and Section \ref{sec:Isolated-and-simple} collects
the definition of blow up points for a sequence of solutions of (\ref{eq:Prob-2})
as well as the definitions of isolated and isolated simple blow up
points. In Section \ref{sec:Blowup-estimates} a careful analysis
of the profile of the rescaled solution near an isolated simple blow
up point is proven. By this result, in Section \ref{sec:Sign-estimates}
we can give an estimate of the sign of the terms of Pohozaev identity
near an isolated simple blow up point. By this result, and by a splitting
Lemma recalled in Section \ref{sec:A-splitting-lemma}, we prove that
only isolated simple blow up points can occur for a sequence of solution
of (\ref{eq:Prob-2}). Finally in Section \ref{sec:Main-Proof} we
will prove that with the hypothesis of Theroem \ref{thm:main}, also
the case of an isolated simple blow up point is ruled out, and we
prove our main result. This strategy of the proof of these compactness
results was firstly introduced by R. Schoen (see \cite{SZ}) and it
is well established in literature, so in this paper we will provide
only the proofs of the new results, while we will give references
for the other ones. 

\subsection{Notations and preliminary definitions}
\begin{nota}
We will use the indices $1\le i,j,k,m,p,r,s\le n-1$ and $1\le a,b,c,d\le n$.
Moreover we use the Einstein convention on repeated indices. We denote
by $g$ the Riemannian metric, by $R_{abcd}$ the full Riemannian
curvature tensor, by $R_{ab}$ the Ricci tensor and by $R_{g}$ and
$h_{g}$ respectively the scalar curvature of $(M,g)$ and the mean
curvature of $\partial M$; moreover the Weyl tensor of $(M,g)$ will
be denoted by $W_{g}$. The bar over an object (e.g. $\bar{W}_{g}$)
will means the restriction to this object to the metric of $\partial M$.

Finally, on the half space $\mathbb{R}_{+}^{n}=\left\{ y=(y_{1},\dots,y_{n-1},y_{n})\in\mathbb{R}^{n},\!\ y_{n}\ge0\right\} $
we set $B_{r}(y_{0})=\left\{ y\in\mathbb{R}^{n},\!\ |y-y_{0}|\le r\right\} $
and $B_{r}^{+}(y_{0})=B_{r}(y_{0})\cap\left\{ y_{n}>0\right\} $.
When $y_{0}=0$ we will use simply $B_{r}=B_{r}(y_{0})$ and $B_{r}^{+}=B_{r}^{+}(y_{0})$.
On the half ball $B_{r}$ we set $\partial'B_{r}^{+}=B_{r}^{+}\cap\partial\mathbb{R}_{+}^{n}=B_{r}^{+}\cap\left\{ y_{n}=0\right\} $
and $\partial^{+}B_{r}^{+}=\partial B_{r}^{+}\cap\left\{ y_{n}>0\right\} $.
On $\mathbb{R}_{+}^{n}$ we will use the following decomposition of
coordinates: $(y_{1},\dots,y_{n-1},y_{n})=(\bar{y},y_{n})=(z,t)$
where $\bar{y},z\in\mathbb{R}^{n-1}$ and $y_{n},t\ge0$.

Fixed a point $q\in\partial M$, we denote by $\psi_{q}:B_{r}^{+}\rightarrow M$
the Fermi coordinates centered at $q$. We denote by $B_{g}^{+}(q,r)$
the image of $\psi_{q}(B_{r}^{+})$. When no ambiguity is possible,
we will denote $B_{g}^{+}(q,r)$ simply by $B_{r}^{+}$, omitting
the chart $\psi_{q}$.
\end{nota}
We introduce the following notation for integral quantities which
recur often in the paper
\[
I_{m}^{\alpha}:=\int_{0}^{\infty}\frac{s^{\alpha}ds}{\left(1+s^{2}\right)^{m}}.
\]
By direct computation (see \cite[Lemma 9.4]{Al}) it holds
\begin{align}
I_{m}^{\alpha}=\frac{2m}{\alpha+1}I_{m+1}^{\alpha+2} & \text{ for }\alpha+1<2m\label{eq:Iam}\\
I_{m}^{\alpha}=\frac{2m}{2m-\alpha-1}I_{m+1}^{\alpha} & \text{ for }\alpha+1<2m\nonumber \\
I_{m}^{\alpha}=\frac{2m-\alpha-3}{\alpha+1}I_{m}^{\alpha+2} & \text{ for }\alpha+3<2m\nonumber 
\end{align}

We shortly recall here the well known function ${\displaystyle U(y):=\frac{1}{\left[(1+y_{n})^{2}+|\bar{y}|^{2}\right]^{\frac{n-2}{2}}}}$
which is also called the standard bubble and which is the unique solution,
up to translations and rescaling, of the nonlinear critical problem.
 
\begin{equation}
\left\{ \begin{array}{ccc}
-\Delta U=0 &  & \text{on }\mathbb{R}_{+}^{n};\\
\frac{\partial U}{\partial y_{n}}=-(n-2)U^{\frac{n}{n-2}} &  & \text{on \ensuremath{\partial}}\mathbb{R}_{+}^{n}.
\end{array}\right.\label{eq:Udelta}
\end{equation}
We set 
\begin{equation}
j_{l}:=\partial_{l}U=-(n-2)\frac{y_{l}}{\left[(1+y_{n})^{2}+|\bar{y}|^{2}\right]^{\frac{n}{2}}}\label{eq:jl}
\end{equation}
\[
\partial_{k}\partial_{l}U=(n-2)\left\{ \frac{ny_{l}y_{k}}{\left[(1+y_{n})^{2}+|\bar{y}|^{2}\right]^{\frac{n+2}{2}}}-\frac{\delta^{kl}}{\left[(1+y_{n})^{2}+|\bar{y}|^{2}\right]^{\frac{n}{2}}}\right\} 
\]
\begin{equation}
j_{n}:=y^{b}\partial_{b}U+\frac{n-2}{2}U=-\frac{n-2}{2}\frac{|y|^{2}-1}{\left[(1+y_{n})^{2}+|\bar{y}|^{2}\right]^{\frac{n}{2}}}.\label{eq:jn}
\end{equation}
and we recall that $j_{1},\dots,j_{n}$ are a base of the space of
the $H^{1}$ solutions of the linearized problem 
\begin{equation}
\left\{ \begin{array}{ccc}
 & -\Delta\phi=0 & \text{on }\mathbb{R}_{+}^{n},\\
 & \frac{\partial\phi}{\partial t}+nU^{\frac{2}{n-2}}\phi=0 & \text{on \ensuremath{\partial}}\mathbb{R}_{+}^{n},\\
 & \phi\in H^{1}(\mathbb{R}_{+}^{n}).
\end{array}\right.\label{eq:linearizzato}
\end{equation}

\section{A Pohozaev type identity\label{sec:Pohozaev}}

In the following, we will use this version of a local Pohozaev type
identity \cite{Al,GM}
\begin{thm}[Pohozaev Identity]
\label{thm:poho} Let $u$ a $C^{2}$-solution of the following problem
\[
\left\{ \begin{array}{cc}
L_{g}u=0 & \text{ in }B_{r}^{+}\\
\frac{\partial u}{\partial\nu}+\frac{n-2}{2}h_{g}u+\varepsilon\alpha u=(n-2)u^{\frac{n}{n-2}} & \text{ on }\partial'B_{r}^{+}
\end{array}\right.
\]
for $B_{r}^{+}=\psi_{q}^{-1}(B_{g}^{+}(q,r))$ for $q\in\partial M$,
with $\tau=\frac{n}{n-2}-p>0$. Let us define
\[
P(u,r):=\int\limits _{\partial^{+}B_{r}^{+}}\left(\frac{n-2}{2}u\frac{\partial u}{\partial r}-\frac{r}{2}|\nabla u|^{2}+r\left|\frac{\partial u}{\partial r}\right|^{2}\right)d\sigma_{r}+\frac{r(n-2)^{2}}{2(n-1)}\int\limits _{\partial(\partial'B_{r}^{+})}u^{\frac{2(n-1)}{n-2}}d\bar{\sigma}_{g},
\]
and
\begin{multline*}
\hat{P}(u,r):=-\int\limits _{B_{r}^{+}}\left(y^{a}\partial_{a}u+\frac{n-2}{2}u\right)[(L_{g}-\Delta)u]dy+\frac{n-2}{2}\int\limits _{\partial'B_{r}^{+}}\left(\bar{y}^{k}\partial_{k}u+\frac{n-2}{2}u\right)h_{g}ud\bar{y}\\
+\frac{n-2}{2}\varepsilon\int\limits _{\partial'B_{r}^{+}}\left(\bar{y}^{k}\partial_{k}u+\frac{n-2}{2}u\right)\alpha ud\bar{y}.
\end{multline*}
Then $P(u,r)=\hat{P}(u,r)$.

Here $a=1,\dots,n$, $k=1,\dots,n-1$ and $y=(\bar{y},y_{n})$, where
$\bar{y}\in\mathbb{R}^{n-1}$ and $y_{n}\ge0$.
\end{thm}

\section{Expansion of the metric\label{sec:Expansion}}

Since the boundary $\partial M$of $M$ is umbilic, given $q\in\partial M$
there exists a conformally related metric $\tilde{g}_{q}=\Lambda_{q}^{\frac{4}{n-2}}g$
such that some geometric quantities at $q$ have a simpler form which
will be summarized in this paragraph. We have 
\[
\Lambda_{q}(q)=1,\ \frac{\partial\Lambda_{q}}{\partial y_{k}}(q)=0\text{ for all }k=1,\dots,n-1.
\]
 Set $\tilde{u}_{q}=\Lambda_{q}^{-1}u$ and problem (\ref{eq:Prob-2})
is equivalent to
\begin{equation}
\left\{ \begin{array}{cc}
L_{\tilde{g}_{q}}\tilde{u}_{q}=0 & \text{ in }M\\
B_{\tilde{g}_{q}}\tilde{u}_{q}+(n-2)\tilde{u}_{q}^{\frac{n}{n-2}}-\varepsilon\left[\Lambda_{q}^{-\frac{2}{n-2}}\alpha\right]\tilde{u}_{q}=0 & \text{ on }\partial M
\end{array}\right..\label{eq:P-conf}
\end{equation}
In the following, in order to simplify notations, we will omit the
\emph{tilda} symbol and we will omit $\psi_{x_{i}}$ whenever is not
needed.
\begin{rem}
\label{rem:confnorm}In Fermi conformal coordinates around $q\in\partial M$,
it holds (see \cite{M1})
\begin{equation}
|\text{det}g_{q}(y)|=1+O(|y|^{n})\label{eq:|g|}
\end{equation}
\begin{eqnarray}
|h_{ij}(y)|=O(|y^{4}|) &  & |h_{g}(y)|=O(|y^{4}|)\label{eq:hij}
\end{eqnarray}
\begin{align}
g_{q}^{ij}(y)= & \delta^{ij}+\frac{1}{3}\bar{R}_{ikjl}y_{k}y_{l}+R_{ninj}y_{n}^{2}\label{eq:gij}\\
 & +\frac{1}{6}\bar{R}_{ikjl,m}y_{k}y_{l}y_{m}+R_{ninj,k}y_{n}^{2}y_{k}+\frac{1}{3}R_{ninj,n}y_{n}^{3}\nonumber \\
 & +\left(\frac{1}{20}\bar{R}_{ikjl,mp}+\frac{1}{15}\bar{R}_{iksl}\bar{R}_{jmsp}\right)y_{k}y_{l}y_{m}y_{p}\nonumber \\
 & +\left(\frac{1}{2}R_{ninj,kl}+\frac{1}{3}\text{Sym}_{ij}(\bar{R}_{iksl}R_{nsnj})\right)y_{n}^{2}y_{k}y_{l}\nonumber \\
 & +\frac{1}{3}R_{ninj,nk}y_{n}^{3}y_{k}+\frac{1}{12}\left(R_{ninj,nn}+8R_{nins}R_{nsnj}\right)y_{n}^{4}+O(|y|^{5})\nonumber 
\end{align}
\begin{equation}
\bar{R}_{g_{q}}(y)=O(|y|^{2})\text{ and }\partial_{ii}^{2}\bar{R}_{g_{q}}=-\frac{1}{6}|\bar{W}|^{2}\label{eq:Rii}
\end{equation}
\begin{equation}
\partial_{tt}^{2}\bar{R}_{g_{q}}=-2R_{ninj}^{2}-2R_{ninj,ij}\label{eq:Rtt}
\end{equation}
\begin{equation}
\bar{R}_{kl}=R_{nn}=R_{nk}=R_{nn,kk}=0\label{eq:Ricci}
\end{equation}
\begin{equation}
R_{nn,nn}=-2R_{nins}^{2}.\label{eq:Rnnnn}
\end{equation}
All the quantities above are calculate in $q\in\partial M$, unless
otherwise specified.
\end{rem}

\section{Isolated and isolated simple blow up points\label{sec:Isolated-and-simple}}

Here we recall the definitions of some type of blow up points, and
we give the basic properties about the behavior of these blow up points
(see \cite{Al,FA,HL,M3}). We will omit the proofs of some well known
results.

Let $\left\{ u_{i}\right\} _{i}$ be a sequence of positive solution
to 
\begin{equation}
\left\{ \begin{array}{cc}
L_{g_{i}}u=0 & \text{ in }M\\
B_{g_{i}}u+(n-2)u^{\frac{n}{n-2}}-\varepsilon_{i}\alpha_{i}u=0 & \text{ on }\partial M
\end{array}\right..\label{eq:Prob-i}
\end{equation}
where $\alpha_{i}=\Lambda_{x_{i}}^{-\frac{2}{n-2}}\alpha\rightarrow\Lambda_{x_{0}}^{-\frac{2}{n-2}}\alpha$,
$x_{i}\rightarrow x_{0}$, $g_{i}\rightarrow g_{0}$ in the $C_{\text{loc}}^{3}$
topology and $0<\varepsilon_{i}<\bar{\varepsilon}$.
\begin{defn}
\label{def:blowup}

1) We say that $x_{0}\in\partial M$ is a blow up point for the sequence
$u_{i}$ of solutions of (\ref{eq:Prob-i}) if there is a sequence
$x_{i}\in\partial M$ of local maxima of $\left.u_{i}\right|_{\partial M}$
such that $x_{i}\rightarrow x_{0}$ and $u_{i}(x_{i})\rightarrow+\infty.$

Shortly we say that $x_{i}\rightarrow x_{0}$ is a blow up point for
$\left\{ u_{i}\right\} _{i}$. 

2) We say that $x_{i}\rightarrow x_{0}$ is an \emph{isolated} blow
up point for $\left\{ u_{i}\right\} _{i}$ if $x_{i}\rightarrow x_{0}$
is a blow up point for $\left\{ u_{i}\right\} _{i}$ and there exist
two constants $\rho,C>0$ such that
\[
u_{i}(x)\le Cd_{\bar{g}}(x,x_{i})^{\frac{2-n}{2}}\text{ for all }x\in\partial M\smallsetminus\left\{ x_{i}\right\} ,\ d_{\bar{g}}(x,x_{i})<\rho.
\]

Given $x_{i}\rightarrow x_{0}$ an isolated blow up point for $\left\{ u_{i}\right\} _{i}$,
and given $\psi_{i}:B_{\rho}^{+}(0)\rightarrow M$ the Fermi coordinates
centered at $x_{i}$, we define the spherical average of $u_{i}$
as
\[
\bar{u}_{i}(r)=\frac{2}{\omega_{n-1}r^{n-1}}\int_{\partial^{+}B_{r}^{+}}u_{i}\circ\psi_{i}d\sigma_{r}
\]
and
\[
w_{i}(r):=r^{\frac{2-n}{2}}\bar{u}_{i}(r)
\]
for $0<r<\rho.$

3) We say that $x_{i}\rightarrow x_{0}$ is an \emph{isolated simple}
blow up point for $\left\{ u_{i}\right\} _{i}$ solutions of (\ref{eq:Prob-i})
if $x_{i}\rightarrow x_{0}$ is an isolated blow up point for $\left\{ u_{i}\right\} _{i}$
and there exists $\rho$ such that $w_{i}$ has exactly one critical
point in the interval $(0,\rho)$.
\end{defn}
Given $x_{i}\rightarrow x_{0}$ a blow up point for $\left\{ u_{i}\right\} _{i}$,
we set
\[
M_{i}:=u_{i}(x_{i})\ \text{ and }\ \delta_{i}:=M_{i}^{\frac{2}{2-n}}.
\]
Obviously $M_{i}\rightarrow+\infty$ and $\delta_{i}\rightarrow0$.

We recall the following results
\begin{prop}
\label{prop:4.1}Let $x_{i}\rightarrow x_{0}$ is an \emph{isolated}
blow up point for $\left\{ u_{i}\right\} _{i}$ and $\rho$ as in
Definition \ref{def:blowup}. We set 
\[
v_{i}(y)=M_{i}^{-1}(u_{i}\circ\psi_{i})(M_{i}^{\frac{2}{2-n}}y),\text{ for }y\in B_{\rho M_{i}^{\frac{n-2}{2}}}^{+}(0).
\]
Then, given $R_{i}\rightarrow\infty$ and $\beta_{i}\rightarrow0$,
up to subsequences, we have
\begin{enumerate}
\item $|v_{i}-U|_{C^{2}\left(B_{R_{i}}^{+}(0)\right)}<\beta_{i}$;
\item ${\displaystyle \lim_{i\rightarrow\infty}\frac{R_{i}}{\log M_{i}}=0}$.
\end{enumerate}
\end{prop}
\begin{prop}
\label{prop:Lemma 4.4}Let $x_{i}\rightarrow x_{0}$ be an isolated
simple blow-up point for $\left\{ u_{i}\right\} _{i}$ and $\alpha<0$.
Let $\eta$ small. Then there exist $C,\rho>0$ such that 
\[
M_{i}^{\lambda_{i}}|\nabla^{k}u_{i}(\psi_{i}(y))|\le C|y|^{2-k-n+\eta}
\]
for $y\in B_{\rho}^{+}(0)\smallsetminus\left\{ 0\right\} $ and $k=0,1,2$.
Here $\lambda_{i}=\left(\frac{2}{n-2}\right)(n-2-\eta)-1$.
\end{prop}
Since $\alpha<0$ the proof of Proposition \ref{prop:Lemma 4.4} is
analogous of Lemma 2.7 of \cite{FA}.
\begin{prop}
\label{prop:eps-i}Let $x_{i}\rightarrow x_{0}$ be an isolated simple
blow-up point for $\left\{ u_{i}\right\} _{i}$ and $\alpha<0$. Then
$\varepsilon_{i}\rightarrow0$ 
\end{prop}
\begin{proof}
We compute the Pohozaev identity in a ball of radius $r$ and we set
$\frac{r}{\delta_{i}}=:R_{i}\rightarrow\infty$. We estimate any term
of $P(u_{i},r_{i})$ and $\hat{P}(u_{i},r_{i})$.

We set 
\begin{align*}
I_{1}(u,r):= & \int\limits _{\partial^{+}B_{r}^{+}}\left(\frac{n-2}{2}u\frac{\partial u}{\partial r}-\frac{r}{2}|\nabla u|^{2}+r\left|\frac{\partial u}{\partial r}\right|^{2}\right)d\sigma_{r}\\
I_{2}(u,r):= & \frac{r(n-2)^{2}}{2(n-1)}\int\limits _{\partial(\partial'B_{r}^{+})}u^{\frac{2(n-1)}{n-2}}d\bar{\sigma}_{g},
\end{align*}
so $P(u_{i},r)=I_{1}(u_{i},r)+I_{2}(u_{i},r)$

By Proposition \ref{prop:Lemma 4.4} we have
\begin{align*}
I_{1}(u_{i},r) & =M_{i}^{-2\lambda_{i}}I_{1}(M_{i}^{\lambda_{i}}u_{i},r)\le cM_{i}^{-2\lambda_{i}}\int\limits _{\partial^{+}B_{r}^{+}}|y|^{2(2-n+\eta)}d\sigma_{r}\le c\delta_{i}^{\lambda_{i}(n-2)}\\
I_{2}(u_{i},r) & \le cM^{-\lambda_{i}\frac{2(n-1)}{n-2}}\le c\delta_{i}^{\lambda_{i}(n-2)}
\end{align*}
Then

\begin{equation}
P(u_{i},r)\le\delta_{i}^{\lambda_{i}(n-2)}.\label{eq:poho1}
\end{equation}
In a similar way we decompose
\begin{align*}
\hat{P}(u,r): & =-\int\limits _{B_{r}^{+}}\left(y^{a}\partial_{a}u+\frac{n-2}{2}u\right)[(L_{g}-\Delta)u]dy+\frac{n-2}{2}\int\limits _{\partial'B_{r}^{+}}\left(\bar{y}^{k}\partial_{k}u+\frac{n-2}{2}u\right)h_{g}ud\bar{y}\\
 & +\frac{n-2}{2}\varepsilon\int\limits _{\partial'B_{r}^{+}}\left(\bar{y}^{k}\partial_{k}u+\frac{n-2}{2}u\right)\alpha ud\bar{y}=:I_{3}(u_{i},r)+I_{4}(u_{i},r)+I_{5}(u_{i},r).
\end{align*}
By Proposition \ref{prop:Lemma 4.4} and by definition of $v_{i}$
we have 
\[
|\nabla^{k}v_{i}(s)|\le M_{i}^{\eta\frac{2}{n-2}}|1+|s||^{2-k-n}=\delta_{i}^{-\eta}|1+|s||^{2-k-n}.
\]
So, after a change of variables, since $|h_{g_{i}}(\delta_{i}s)|\le O(\delta_{i}^{4}|s|^{4})$,
\begin{equation}
|I_{4}(u_{i},r)|=\frac{n-2}{2}\delta_{i}\int\limits _{\partial'B_{R_{i}}^{+}}\left(\bar{s}^{k}\partial_{k}v_{i}+\frac{n-2}{2}v_{i}\right)h_{g_{i}}(\delta_{i}s)v_{i}d\bar{s}\le c\delta_{i}^{5-2\eta}.\label{eq:poho3}
\end{equation}
Analogously 
\[
I_{5}(u_{i},r)=\varepsilon_{i}\delta_{i}\int\limits _{\partial'B_{R_{i}}^{+}}\left(\bar{s}^{k}\partial_{k}v_{i}+\frac{n-2}{2}v_{i}\right)\alpha_{i}(\delta_{i}s)v_{i}d\bar{s}.
\]
Since $\alpha_{i}(\delta_{i}s)=\Lambda_{x_{i}}^{\frac{2}{2-n}}(\delta_{i}s)\alpha_{i}(\delta_{i}s)$
and by Claim 1 of Proposition \ref{prop:4.1} and (\ref{eq:avj})
we get
\begin{align}
\lim_{i\rightarrow\infty} & \int\limits _{\partial'B_{R_{i}}^{+}}\left(\bar{s}^{k}\partial_{k}v_{i}+\frac{n-2}{2}v_{i}\right)\alpha_{i}(\delta_{i}s)v_{i}d\bar{s}\label{eq:poho4}\\
 & =\alpha(x_{0})\int\limits _{\partial'B_{R_{i}}^{+}}\left(\bar{s}^{k}\partial_{k}U+\frac{n-2}{2}U\right)Ud\bar{s}\nonumber \\
 & =\frac{n-2}{2}\alpha(x_{0})\int\limits _{\mathbb{R}^{n-1}}\frac{1-|\bar{s}|^{2}}{\left[1+|\bar{s}|^{2}\right]^{n-1}}d\bar{s}=:A>0.\nonumber 
\end{align}
Furthermore we have 
\[
I_{3}(u_{i},r)=-\int\limits _{B_{r}^{+}}\left(s^{a}\partial_{a}v_{i}+\frac{n-2}{2}v\right)[(L_{\hat{g}}-\Delta)v_{i}]dy
\]
and it holds 
\[
\left(L_{\hat{g}}-\Delta\right)v=\left(g^{kl}(\delta_{i}s)-\delta^{kl}\right)\partial_{kl}v+\delta_{i}\partial_{k}g^{kl}(\delta_{i}s)\partial_{l}v-\delta_{i}^{2}\frac{n-2}{4(n-1)}R_{g}(\delta_{i}s)v+O(\delta_{i}^{N}|s|^{N-1})\partial_{l}v
\]
we have 
\begin{equation}
|I_{3}(u_{i},r)|\le c\delta_{i}^{2-2\eta}\label{eq:poho2}
\end{equation}

Concluding, by (\ref{eq:poho1}), (\ref{eq:poho2}), (\ref{eq:poho3}),
(\ref{eq:poho4}) we get 
\[
-c\delta_{i}^{2-2\eta}+(A+o(1))\varepsilon_{i}\delta_{i}\le\delta_{i}^{\lambda_{i}(n-2)}
\]
 which is possible only if $\varepsilon_{i}\rightarrow0$.
\end{proof}
Since $\varepsilon_{i}\rightarrow0$ by Prop. \ref{prop:eps-i}, the
proof of the next proposition is analogous to Prop. 4.3 in \cite{Al}
\begin{prop}
\label{prop:4.3}Let $x_{i}\rightarrow x_{0}$ be an isolated simple
blow-up point for $\left\{ u_{i}\right\} _{i}$ and $\alpha<0$. Then
there exist $C,\rho>0$ such that 
\begin{enumerate}
\item $M_{i}u_{i}(\psi_{i}(y))\le C|y|^{2-n}$ for all $y\in B_{\rho}^{+}(0)\smallsetminus\left\{ 0\right\} $;
\item $M_{i}u_{i}(\psi_{i}(y))\ge C^{-1}G_{i}(y)$ for all $y\in B_{\rho}^{+}(0)\smallsetminus B_{r_{i}}^{+}(0)$
where $r_{i}:=R_{i}M_{i}^{\frac{2}{2-n}}$ and $G_{i}$ is the Green\textquoteright s
function which solves
\[
\left\{ \begin{array}{ccc}
L_{g_{i}}G_{i}=0 &  & \text{in }B_{\rho}^{+}(0)\smallsetminus\left\{ 0\right\} \\
G_{i}=0 &  & \text{on }\partial^{+}B_{\rho}^{+}(0)\\
B_{g_{i}}G_{i}=0 &  & \text{on }\partial'B_{\rho}^{+}(0)\smallsetminus\left\{ 0\right\} 
\end{array}\right.
\]
\end{enumerate}
and $|y|^{n-2}G_{i}(y)\rightarrow1$ as $|z|\rightarrow0$.
\end{prop}
By Proposition \ref{prop:4.1} and Proposition \ref{prop:4.3} we
have that, if $x_{i}\rightarrow x_{0}$ is an isolated simple blow-up
point for $\left\{ u_{i}\right\} _{i}$, then it holds
\[
v_{i}\le CU\text{ in }B_{\rho M_{i}^{\frac{2}{2-n}}}^{+}(0).
\]

\section{Blowup estimates\label{sec:Blowup-estimates}}

Our aim is to provide a fine estimate for the approximation of the
rescaled solution near an isolated simple blow up point.

In the following lemma, given a point $q\in\partial M$, we introduce
the function $\gamma_{q}$ which arises from the secondo order term
of the expansion of the metric $g$ on $M$ (see \ref{eq:gij}). The
choice of this function plays a fundamental role in this paper. Using
the function $\gamma_{q}$ we are able to cancel the term of second
order in formula (\ref{eq:Qi-parziale}). Also, the estimates of Proposition
\ref{prop:stimawi} and of Lemma \ref{lem:R(U,U)} depend on the properties
of function $\gamma_{q}$.

For the proof of the Lemma we refer to \cite[Lemma 3]{GMP} and \cite[Proposition 5.1]{Al}.
\begin{lem}
\label{lem:vq}Assume $n\ge5$. Given a point $q\in\partial M$, there
exists a unique $\gamma_{q}:\mathbb{R}_{+}^{n}\rightarrow\mathbb{R}$
a solution of the linear problem 
\begin{equation}
\left\{ \begin{array}{ccc}
-\Delta\gamma=\left[\frac{1}{3}\bar{R}_{ijkl}(q)y_{k}y_{l}+R_{ninj}(q)y_{n}^{2}\right]\partial_{ij}^{2}U &  & \text{on }\mathbb{R}_{+}^{n}\\
\frac{\partial\gamma}{\partial y_{n}}=-nU^{\frac{2}{n-2}}\gamma &  & \text{on }\partial\mathbb{R}_{+}^{n}
\end{array}\right.\label{eq:vqdef}
\end{equation}
which is $L^{2}(\mathbb{R}_{+}^{n})$-orthogonal to the functions
$j_{1},\dots,j_{n}$ defined in (\ref{eq:jl}) and (\ref{eq:jn}).

Moreover it holds
\begin{equation}
|\nabla^{\tau}\gamma_{q}(y)|\le C(1+|y|)^{4-\tau-n}\text{ for }\tau=0,1,2.\label{eq:gradvq}
\end{equation}
\begin{equation}
\int_{\mathbb{R}_{+}^{n}}\gamma_{q}\Delta\gamma_{q}dy\le0,\label{new}
\end{equation}

\begin{equation}
\int_{\partial\mathbb{R}_{+}^{n}}U^{\frac{n}{n-2}}(t,z)\gamma_{q}(t,z)dz=0\label{eq:Uvq}
\end{equation}
\begin{equation}
\gamma_{q}(0)=\frac{\partial\gamma_{q}}{\partial y_{1}}(0)=\dots=\frac{\partial\gamma_{q}}{\partial y_{n-1}}(0)=0.\label{eq:dervq}
\end{equation}

Finally the map $q\mapsto\gamma_{q}$ is $C^{2}(\partial M)$.
\end{lem}
In this section $x_{i}\rightarrow x_{0}$ is an isolated simple blowup
point for a sequence $\left\{ u_{i}\right\} _{i}$ of solutions of
(\ref{eq:Prob-i}). We will work in the conformal Fermi coordinates
in a neighborhood of $x_{i}$.

Set $\tilde{u}_{i}=\Lambda_{x_{i}}^{-1}u_{i}$ and
\begin{equation}
\delta_{i}:=\tilde{u}_{i}^{\frac{2}{2-n}}(x_{i})=u_{i}^{\frac{2}{2-n}}(x_{i})=M_{i}^{\frac{2}{2-n}}\ \ \ v_{i}(y):=\delta_{i}^{\frac{n-2}{2}}u_{i}(\delta_{i}y)\text{ for }y\in B_{\frac{R}{\delta_{i}}}^{+}(0).\label{eq:deltai}
\end{equation}
 Then $v_{i}$ satisfies 
\begin{equation}
\left\{ \begin{array}{cc}
L_{\hat{g}_{i}}v_{i}=0 & \text{ in }B_{\frac{R}{\delta_{i}}}^{+}(0)\\
B_{\hat{g}_{i}}v_{i}+(n-2)v_{i}^{\frac{n}{n-2}}-\varepsilon_{i}\alpha_{i}(\delta_{i}y)v_{i}=0 & \text{ on }\partial'B_{\frac{R}{\delta_{i}}}^{+}(0)
\end{array}\right.\label{eq:Prob-hat}
\end{equation}
 where $\hat{g}_{i}:=\tilde{g}_{i}(\delta_{i}y)=\Lambda_{x_{i}}^{\frac{4}{n-2}}(\delta_{i}y)g(\delta_{i}y)$,
and $\alpha_{i}(y)=\Lambda_{x_{i}}^{-\frac{2}{n-2}}(y)\alpha(y)$. 

The estimates that follow are similar to the ones of \cite[Lemma 6.1]{Al}
and \cite[Section 4]{GM}, where the main difference is the term containing
the linear perturbation $\alpha$. For the sake of self-containedness
we sketch the main proofs.
\begin{lem}
\label{lem:coreLemma}Assume $n\ge8$. Let $\gamma_{x_{i}}$ be defined
in (\ref{eq:vqdef}). There exist $R,C>0$ such that 
\[
|v_{i}(y)-U(y)-\delta_{i}^{2}\gamma_{x_{i}}(y)|\le C\left(\delta_{i}^{3}+\varepsilon_{i}\delta_{i}\right)
\]
for $|y|\le R/\delta_{i}$.
\end{lem}
\begin{proof}
Let $y_{i}$ such that 
\[
\mu_{i}:=\max_{|y|\le R/\delta_{i}}|v_{i}(y)-U(y)-\delta_{i}^{2}\gamma_{x_{i}}(y)|=|v_{i}(y_{i})-U(y_{i})-\delta_{i}^{2}\gamma_{x_{i}}(y_{i})|.
\]
We can assume, without loss of generality, that $|y_{i}|\le\frac{R}{2\delta_{i}}.$

In fact, suppose that there exists $c>0$ such that $|y_{i}|>\frac{c}{\delta_{i}}$
for all $i$. Then, since $v_{i}(y)\le CU(y)$, and by (\ref{eq:gradvq}),
we get the inequality
\[
|v_{i}(y_{i})-U(y_{i})-\delta_{i}^{2}\gamma_{x_{i}}(y_{i})|\le C\left(|y_{i}|^{2-n}+\delta_{i}^{2}|y_{i}|^{4-n}\right)\le C\delta_{i}^{n-2}
\]
which proves the Lemma. So, in the next we will suppose $|y_{i}|\le\frac{R}{2\delta_{i}}$.
This fact will be used later.

By contradiction, suppose that 
\begin{equation}
\max\left\{ \mu_{i}^{-1}\delta_{i}^{3},\mu_{i}^{-1}\varepsilon_{i}\delta_{i}\right\} \rightarrow0\text{ when }i\rightarrow\infty.\label{eq:ipass}
\end{equation}
Defined 
\[
w_{i}(y):=\mu_{i}^{-1}\left(v_{i}(y)-U(y)-\delta_{i}^{2}\gamma_{x_{i}}(y)\right)\text{ for }|y|\le R/\delta_{i},
\]
we have, by direct computation, that $w_{i}$ satisfies 
\begin{equation}
\left\{ \begin{array}{cc}
L_{\hat{g}_{i}}w_{i}=Q_{i} & \text{ in }B_{\frac{R}{\delta_{i}}}^{+}(0)\\
B_{\hat{g}_{i}}w_{i}+b_{i}w_{i}={F}_{i} & \text{ on }\partial'B_{\frac{R}{\delta_{i}}}^{+}(0)
\end{array}\right.\label{eq:wi}
\end{equation}
where 
\begin{align*}
b_{i}= & (n-2)\frac{v_{i}^{\frac{n}{n-2}}-(U+\delta_{i}^{2}\gamma_{x_{i}})^{\frac{n}{n-2}}}{v_{i}-U-\delta_{i}^{2}\gamma_{x_{i}}}\\
\bar{Q}_{i}= & -\frac{1}{\mu_{i}}\left\{ (n-2)(U+\delta_{i}^{2}\gamma_{x_{i}})^{\frac{n}{n-2}}-(n-2)U^{\frac{n}{n-2}}-n\delta_{i}^{2}U^{\frac{2}{n-2}}\gamma_{x_{i}}-\frac{n-2}{2}h_{\hat{g}_{i}}(U+\delta_{i}^{2}\gamma_{x_{i}})\right\} \\
F_{i}= & \bar{Q}_{i}+\frac{\varepsilon_{i}\delta_{i}}{\mu_{i}}\alpha_{i}(\delta_{i}y)v_{i}(y)\\
Q_{i}= & -\frac{1}{\mu_{i}}\left\{ \left(L_{\hat{g}_{i}}-\Delta\right)(U+\delta_{i}^{2}\gamma_{x_{i}})+\delta_{i}^{2}\Delta\gamma_{x_{i}}\right\} .
\end{align*}
We estimate for terms $b_{i},Q_{i,}F_{i}$ obtaining that the sequence
$w_{i}$ converges in $C_{\text{loc}}^{2}(\mathbb{R}_{+}^{n})$ to
some $w$ solution of 
\begin{equation}
\left\{ \begin{array}{cc}
\Delta w=0 & \text{ in }\mathbb{R}_{+}^{n}\\
\frac{\partial}{\partial\nu}w+nU^{\frac{n}{n-2}}w=0 & \text{ on }\partial\mathbb{R}_{+}^{n}
\end{array}\right.,\label{eq:diff-w}
\end{equation}
then we will derive a contradiction using (\ref{eq:ipass}). 

Since $v_{i}\rightarrow U$ in $C_{\text{loc}}^{2}(\mathbb{R}_{+}^{n})$
we have, at once, 
\begin{align}
b_{i} & \rightarrow nU^{\frac{2}{n-2}}\text{ in }C_{\text{loc}}^{2}(\mathbb{R}_{+}^{n})\label{eq:b1}\\
|b_{i}(y)| & \le(1+|y|)^{-2}\text{ for }|y|\le R/\delta_{i}.\label{eq:b2}
\end{align}
We proceed now by estimating $Q_{i}$ and $\bar{Q}_{i}$. We recall
that 
\begin{align}
[L_{\hat{g}_{i}}-\Delta]u(y)= & \left(g_{i}^{kl}(\delta_{i}y)-\delta^{kl}\right)\partial_{k}\partial_{l}u+\delta_{i}\partial_{k}g_{i}^{kl}(\delta_{i}y)\partial_{l}u-\delta_{i}^{2}\frac{n-2}{4(n-1)}R_{g_{i}}(\delta_{i}y)u\nonumber \\
 & +O(\delta_{i}^{N}|y|^{N-1})\partial_{l}u\label{eq:L-Delta}
\end{align}
where $N$ can be chosen arbitrarily large. At this point using the
definition of the function $\gamma_{x_{i}}$, (\ref{eq:L-Delta}),
(\ref{eq:gij}) and the decays properties of $U$ and $\gamma_{x_{i}}$
we obtain
\begin{align}
-\mu_{i}Q_{i}= & \delta_{i}^{2}\left(\frac{1}{3}\bar{R}_{kslj}y_{s}y_{j}+R_{nkns}y_{n}^{2}\right)\left(\partial_{k}\partial_{l}U+\delta_{i}^{2}\partial_{k}\partial_{l}\gamma_{x_{i}}\right)\nonumber \\
 & +O(\delta_{i}^{3}|y|^{3})\left(\partial_{k}\partial_{l}U+\delta_{i}^{2}\partial_{k}\partial_{l}\gamma_{x_{i}}\right)\nonumber \\
 & +\delta_{i}^{2}\left(\frac{1}{3}\bar{R}_{kklj}y_{j}+\frac{1}{3}\bar{R}_{kslk}y_{s}\right)\left(\partial_{l}U+\delta_{i}^{2}\partial_{l}\gamma_{x_{i}}\right)\nonumber \\
 & +O(\delta_{i}^{3}|y|^{2})\left(\partial_{l}U+\delta_{i}^{2}\partial_{l}\gamma_{x_{i}}\right)\nonumber \\
 & +O(\delta_{i}^{4}|y|^{2})\left(U+\delta_{i}^{2}\gamma_{x_{i}}\right)\nonumber \\
 & +\delta_{i}^{2}\Delta\gamma_{x_{i}}+O(\delta_{i}^{N}|y|^{N-1})\left(\partial_{l}U+\delta_{i}^{2}\partial_{l}\gamma_{x_{i}}\right)\nonumber \\
= & O\left(\delta_{i}^{3}\left(1+|y|\right)^{3-n}\right)+O\left(\delta_{i}^{4}\left(1+|y|\right)^{4-n}\right)+O\left(\delta_{i}^{5}\left(1+|y|\right)^{5-n}\right)\nonumber \\
 & +O\left(\delta_{i}^{6}\left(1+|y|\right)^{6-n}\right)+O\left(\delta_{i}^{N}\left(1+|y|\right)^{N-n}\right)O\left(\delta_{i}^{N+2}\left(1+|y|\right)^{N+2-n}\right).\label{eq:Qi-parziale}
\end{align}
Since $|y|\le R/\delta_{i}$, we have $\delta_{i}\left(1+|y|\right)\le C$,
thus 
\begin{equation}
Q_{i}=O(\mu_{i}^{-1}\delta_{i}^{3}\left(1+|y|\right)^{3-n}).\label{eq:Q}
\end{equation}
In light of (\ref{eq:ipass}) we have also $Q_{i}\in L^{p}(B_{R/\delta_{i}}^{+})$
for all $p\ge2$.

By Taylor expansion, and proceeding as above we have
\begin{align*}
-\mu_{i}\bar{Q}_{i,1}= & \left\{ \delta_{i}^{4}\frac{2}{n-2}(U+\theta\delta_{i}^{2}\gamma_{x_{i}})^{\frac{4-n}{n-2}}\gamma_{x_{i}}^{2}-\delta_{i}\frac{n-2}{2}h_{g_{i}}(\delta_{i}y)(U+\delta_{i}^{2}\gamma_{x_{i}})\right\} \\
= & O(\delta_{i}^{4}\left(1+|y|\right)^{5-n}).
\end{align*}
Since $|v_{i}(y)|\le CU(y)$ we have
\begin{align}
F_{i} & =\bar{Q_{i}}+O(\mu_{i}^{-1}\varepsilon_{i}\delta_{i}\left(1+|y|\right)^{2-n})\label{eq:Qbar}\\
 & =O(\mu_{i}^{-1}\delta_{i}^{4}\left(1+|y|\right)^{5-n})+O(\mu_{i}^{-1}\varepsilon_{i}\delta_{i}\left(1+|y|\right)^{2-n}),\nonumber 
\end{align}
and $F_{i}\in L^{p}(\partial'B_{R/\delta_{i}}^{+})$ for all $p\ge2$. 

Finally we remark that $|w_{i}(y)|\le1$, so by (\ref{eq:ipass})
(\ref{eq:b1}), (\ref{eq:b2}), (\ref{eq:Q}), (\ref{eq:Qbar}) and
by standard elliptic estimates we conclude that, up to subsequence,
$\left\{ w_{i}\right\} _{i}$ converges in $C_{\text{loc}}^{2}(\mathbb{R}_{+}^{n})$
to some $w$ solution of (\ref{eq:diff-w}).

The next step is to prove that $|w(y)|\le C(1+|y|^{-1})$ for $y\in\mathbb{R}_{+}^{n}$.
Consider $G_{i}$ the Green function for the conformal Laplacian $L_{\hat{g}_{i}}$
defined on $B_{r/\delta_{i}}^{+}$ with boundary conditions $B_{\hat{g}_{i}}G_{i}=0$
on $\partial'B_{r/\delta_{i}}^{+}$ and $G_{i}=0$ on $\partial^{+}B_{r/\delta_{i}}^{+}$.
It is well known that $G_{i}=O(|\xi-y|^{2-n})$. By the Green formula
and by (\ref{eq:Q}) and (\ref{eq:Qbar}) we have
\begin{align*}
w_{i}(y)= & -\int_{B_{\frac{R}{\delta_{i}}}^{+}}G_{i}(\xi,y)Q_{i}(\xi)d\mu_{\hat{g}_{i}}(\xi)-\int_{\partial^{+}B_{\frac{R}{\delta_{i}}}^{+}}\frac{\partial G_{i}}{\partial\nu}(\xi,y)w_{i}(\xi)d\sigma_{\hat{g}_{i}}(\xi)\\
 & +\int_{\partial'B_{\frac{R}{\delta_{i}}}^{+}}G_{i}(\xi,y)\left(b_{i}(\xi)w_{i}(\xi)-F_{i}(\xi)\right)d\sigma_{\hat{g}_{i}}(\xi),
\end{align*}
so 
\begin{align*}
|w_{i}(y)| & \le\frac{\delta_{i}^{3}}{\mu_{i}}\int_{B_{\frac{R}{\delta_{i}}}^{+}}|\xi-y|^{2-n}(1+|\xi|)^{3-n}d\xi+\int_{\partial^{+}B_{\frac{R}{\delta_{i}}}^{+}}|\xi-y|^{1-n}w_{i}(\xi)d\sigma(\xi)\\
 & +\int_{\partial'B_{\frac{R}{\delta_{i}}}^{+}}|\bar{\xi}-y|^{2-n}\left((1+|\bar{\xi}|)^{-2}+\frac{\delta_{i}^{4}}{\mu_{i}}(1+|\bar{\xi}|)^{5-n}+\frac{\varepsilon_{i}\delta_{i}}{\mu_{i}}(1+|\bar{\xi}|)^{2-n}\right)d\bar{\xi},
\end{align*}
Notice that in the second integral we used that $|y|\le\frac{R}{2\delta_{i}}$
to estimate $|\xi-y|\ge|\xi|-|y|\ge\frac{R}{2\delta_{i}}$ on $\partial^{+}B_{R/\delta_{i}}^{+}$.
Moreover, since $v_{i}(\xi)\le CU(\xi)$, we get
\begin{equation}
|w_{i}(\xi)|\le\frac{C}{\mu_{i}}\left(\left(1+|\xi|\right)^{2-n}+\delta_{i}^{2}\left(1+|\xi|\right)^{4-n}\right)\le C\frac{\delta_{i}^{n-2}}{\mu_{i}}\text{ on }\partial^{+}B_{R/\delta_{i}}^{+};\label{eq:wibordo}
\end{equation}
hence 
\begin{equation}
\int_{\partial^{+}B_{\frac{R}{\delta_{i}}}^{+}}|\xi-y|^{1-n}w_{i}(\xi)d\sigma(\xi)\le C\int_{\partial^{+}B_{\frac{R}{\delta_{i}}}^{+}}\frac{\delta_{i}^{2n-3}}{\mu_{i}}d\sigma_{\hat{g}_{i}}(\xi)\le C\frac{\delta_{i}^{n-2}}{\mu_{i}}.\label{eq:stimaW1}
\end{equation}
 For the other terms we use the following formula (see \cite[Lemma 9.2]{Al}
and \cite{Au,Gi}) 
\begin{equation}
\int_{\mathbb{R}^{m}}|\xi-y|^{\beta-m}(1+|y|)^{-\eta}\le C(1+|y|)^{\beta-\eta}\label{eq:ALstimagreen}
\end{equation}
where $y\in\mathbb{R}^{m+k}\supseteq\mathbb{R}^{m}$, $\eta,\beta\in\mathbb{N}$,
$0<\beta<\eta<m$. We get
\begin{equation}
\frac{\delta_{i}^{3}}{\mu_{i}}\int_{B_{\frac{R}{\delta_{i}}}^{+}}|\xi-y|^{2-n}(1+|\xi|)^{3-n}d\xi\le C\frac{\delta_{i}^{3}}{\mu_{i}}(1+|y|)^{5-n},\label{eq:stimaW2}
\end{equation}

\begin{equation}
\int_{\partial'B_{\frac{R}{\delta_{i}}}^{+}}|\bar{\xi}-y|^{2-n}(1+|\bar{\xi}|)^{-2}d\bar{\xi}\le(1+|y|)^{-1}\label{eq:stimaW3}
\end{equation}
\begin{equation}
\frac{\delta_{i}^{4}}{\mu_{i}}\int_{\partial'B_{\frac{R}{\delta_{i}}}^{+}}|\bar{\xi}-y|^{2-n}(1+|\bar{\xi}|)^{5-n}d\bar{\xi}\le C\frac{\delta_{i}^{4}}{\mu_{i}}(1+|y|)^{6-n}\label{eq:stimaW4}
\end{equation}

\begin{equation}
\frac{\varepsilon_{i}\delta_{i}}{\mu_{i}}\int_{\partial'B_{\frac{R}{\delta_{i}}}^{+}}|\bar{\xi}-y|^{2-n}(1+|\bar{\xi}|)^{2-n}d\bar{\xi}\le C\frac{\varepsilon_{i}\delta_{i}}{\mu_{i}}(1+|y|)^{3-n}.\label{eq:stimaW5}
\end{equation}
By (\ref{eq:stimaW1}), (\ref{eq:stimaW2}), (\ref{eq:stimaW3}),
(\ref{eq:stimaW4}) (\ref{eq:stimaW5}) we have 
\begin{equation}
|w_{i}(y)|\le C\left((1+|y|)^{-1}+\frac{\delta_{i}^{3}}{\mu_{i}}(1+|y|)^{5-n}+\frac{\varepsilon_{i}\delta_{i}}{\mu_{i}}(1+|y|)^{3-n}\right)\text{ for }|y|\le\frac{R}{2\delta_{i}}\label{eq:stimaWiass}
\end{equation}
so by assumption (\ref{eq:ipass}) we prove 
\begin{equation}
|w(y)|\le C(1+|y|)^{-1}\text{ for }y\in\mathbb{R}_{+}^{n}\label{eq:stimaWass}
\end{equation}
as claimed.

Finally we notice that, since $v_{i}\rightarrow U$ near $0$, and
by (\ref{eq:dervq}) we have $w_{i}(0)\rightarrow0$ as well as $\frac{\partial w_{i}}{\partial y_{j}}(0)\rightarrow0$
for $j=1,\dots,n-1$. This implies that 
\begin{equation}
w(0)=\frac{\partial w}{\partial y_{1}}(0)=\dots=\frac{\partial w}{\partial y_{n-1}}(0)=0.\label{eq:W(0)}
\end{equation}
We are ready now to prove the contradiction. In fact, it is known
(see \cite[Lemma 2]{Al}) that any solution of (\ref{eq:diff-w})
that decays as (\ref{eq:stimaWass}) is a linear combination of $\frac{\partial U}{\partial y_{1}},\dots,\frac{\partial U}{\partial y_{n-1}},\frac{n-2}{2}U+y^{b}\frac{\partial U}{\partial y_{b}}$.
This fact, combined with (\ref{eq:W(0)}), implies that $w\equiv0$. 

Now, on one hand $|y_{i}|\le\frac{R}{2\delta_{i}}$, so estimate (\ref{eq:stimaWiass})
holds; on the other hand, since $w_{i}(y_{i})=1$ and $w\equiv0$,
we get $|y_{i}|\rightarrow\infty$, obtaining
\[
1=w_{i}(y_{i})\le C(1+|y_{i}|)^{-1}\rightarrow0
\]
 which gives us the contradiction. 
\end{proof}
\begin{lem}
\label{lem:taui}Assume $n\ge8$ and $\alpha<0$. There exists $C>0$
such that 
\[
\varepsilon_{i}\delta_{i}\le C\delta_{i}^{3}.
\]
\end{lem}
\begin{proof}
We proceed by contradiction, supposing that 
\begin{equation}
\left(\varepsilon_{i}\delta_{i}\right)^{-1}\delta_{i}^{3}\rightarrow0\text{ when }i\rightarrow\infty.\label{eq:ipasstau}
\end{equation}
Thus, by Lemma \ref{lem:coreLemma}, we have 
\[
|v_{i}(y)-U(y)-\delta_{i}^{2}\gamma_{x_{i}}(y)|\le C\varepsilon_{i}\delta_{i}\text{ for }|y|\le R/\delta_{i}.
\]
We define, similarly to Lemma \ref{lem:coreLemma},
\[
w_{i}(y):=\frac{1}{\varepsilon_{i}\delta_{i}}\left(v_{i}(y)-U(y)-\delta_{i}^{2}\gamma_{x_{i}}(y)\right)\text{ for }|y|\le R/\delta_{i},
\]
and we have that $w_{i}$ satisfies (\ref{eq:wi}) where 
\begin{align*}
b_{i}= & (n-2)\frac{v_{i}^{\frac{n}{n-2}}-(U+\delta_{i}^{2}\gamma_{x_{i}})^{\frac{n}{n-2}}}{v_{i}-U-\delta_{i}^{2}\gamma_{x_{i}}}\\
\bar{Q}_{i}= & -\frac{1}{\varepsilon_{i}\delta_{i}}\left\{ (n-2)(U+\delta_{i}^{2}\gamma_{x_{i}})^{\frac{n}{n-2}}-(n-2)U^{\frac{n}{n-2}}-n\delta_{i}^{2}U^{\frac{2}{n-2}}\gamma_{x_{i}}-\frac{n-2}{2}h_{\hat{g}_{i}}(U+\delta_{i}^{2}\gamma_{x_{i}})\right\} \\
F_{i}= & \bar{Q}_{i}+\alpha_{i}(\delta_{i}y)v_{i}(y)\\
Q_{i}= & -\frac{1}{\varepsilon_{i}\delta_{i}}\left\{ \left(L_{\hat{g}_{i}}-\Delta\right)(U+\delta_{i}^{2}\gamma_{x_{i}})+\delta_{i}^{2}\Delta\gamma_{x_{i}}\right\} .
\end{align*}
As before, $b_{i}$ satisfies inequality (\ref{eq:b2}) while 
\begin{align}
Q_{i} & =O\left(\left(\varepsilon_{i}\delta_{i}\right)^{-1}\delta_{i}^{3}\left(1+|y|\right)^{3-n}\right)\label{eq:Q-1-1}\\
\bar{Q}_{i} & =O\left(\left(\varepsilon_{i}\delta_{i}\right)^{-1}\delta_{i}^{4}\left(1+|y|\right)^{5-n}\right),\label{eq:Q-1-bar}\\
F_{i} & =O\left(\left(\varepsilon_{i}\delta_{i}\right)^{-1}\delta_{i}^{4}\left(1+|y|\right)^{5-n}\right)+O(\left(1+|y|\right)^{2-n}),\label{eq:Q-1-2}
\end{align}
so by classic elliptic estimates we can prove that the sequence $w_{i}$
converges in $C_{\text{loc}}^{2}(\mathbb{R}_{+}^{n})$ to some $w$.

Moreover, we can proceed as in Lemma \ref{lem:coreLemma} to deduce
that 
\begin{align}
|w_{i}(y)| & \le C\left((1+|y|)^{-1}+\frac{\delta_{i}^{3}}{\varepsilon_{i}\delta_{i}}(1+|y|)^{5-n}+\left(1+|y|\right)^{3-n}\right)\label{eq:decWtau}\\
 & \le C\left((1+|y|)^{-1}+\frac{\delta_{i}^{3}}{\varepsilon_{i}\delta_{i}}(1+|y|)^{5-n}\right)\text{ for }|y|\le\frac{R}{2\delta_{i}}.\nonumber 
\end{align}

Now let $j_{n}$ defined as in (\ref{eq:jn}). In light of (\ref{eq:Q-1-bar})
easily we get 
\[
\lim_{i\rightarrow+\infty}\int_{\partial'B_{\frac{R}{\delta_{i}}}^{+}}j_{n}\bar{Q}_{i}d\sigma_{\hat{g}_{i}}=0.
\]
We recall that $\alpha_{i}(\delta_{i}y)=\Lambda_{x_{i}}^{-\frac{2}{n-2}}(\delta_{i}y)\alpha(\delta_{i}y)$,
so, by Proposition \ref{prop:4.1}, we have 
\[
\alpha_{i}(\delta_{i}y)v_{i}(y)\rightarrow\alpha(x_{0})U(y)\text{ for }i\rightarrow+\infty.
\]
So, since $\alpha<0$, we get, by (\ref{eq:Iam}),
\begin{multline}
\lim_{i\rightarrow+\infty}\int_{\partial'B_{\frac{R}{\delta_{i}}}^{+}}\alpha_{i}(\delta_{i}y)v_{i}(y)j_{n}(y)
=\alpha(x_{0})\int_{\mathbb{R}^{n-1}}\frac{1-|\bar{y}|^{2}}{\left(1+|\bar{y}|^{2}\right)^{n-1}}\\
=\alpha(x_{0})\omega_{n-2}\int_{0}^{+\infty}\frac{s^{n-2}-s^{n}}{(1+s^{2})^{n-1}}ds=-\frac{2\omega_{n-2}}{n-1}\alpha(x_{0})\int_{0}^{+\infty}\frac{s^{n}}{(1+s^{2})^{n-1}}ds>0,\label{eq:avj}
\end{multline}
where $\omega_{n-2}$ is the volume element of the $(n-1)$ unit sphere
and where we used (\ref{eq:Iam}) in the last passage.  Thus we have
\begin{equation}
\lim_{i\rightarrow+\infty}\int_{\partial'B_{\frac{R}{\delta_{i}}}^{+}}j_{n}F_{i}d\sigma_{\hat{g}_{i}}>0,\label{eq:neg}
\end{equation}
and (\ref{eq:neg}) leads us to a contradiction. Indeed, since $w_{i}$
satisfies (\ref{eq:wi}), integrating by parts we obtain
\begin{multline*}
\int_{\partial'B_{\frac{R}{\delta_{i}}}^{+}}j_{n}F_{i}d\sigma_{\hat{g}_{i}}=\int_{\partial'B_{\frac{R}{\delta_{i}}}^{+}}j_{n}\left[B_{\hat{g}_{i}}w_{i}+b_{i}w_{i}\right]d\sigma_{\hat{g}_{i}}\\
=\int_{\partial'B_{\frac{R}{\delta_{i}}}^{+}}w_{i}\left[B_{\hat{g}_{i}}j_{n}+b_{i}j_{n}\right]d\sigma_{\hat{g}_{i}}+\int_{\partial^{+}B_{\frac{R}{\delta_{i}}}^{+}}\left[\frac{\partial j_{n}}{\partial\eta_{i}}w_{i}-\frac{\partial w_{i}}{\partial\eta_{i}}j_{n}\right]d\sigma_{\hat{g}_{i}}\\
+\int_{B_{\frac{R}{\delta_{i}}}^{+}}\left[w_{i}L_{\hat{g}_{i}}j_{n}-j_{n}L_{\hat{g}_{i}}w_{i}\right]d\mu_{\hat{g}_{i}}
\end{multline*}
where $\eta_{i}$ is the inward unit normal vector to $\partial^{+}B_{\frac{R}{\delta_{i}}}^{+}$. 

By the decay of $j_{n}$ and by the decay of $w_{i}$, given by (\ref{eq:decWtau})
and by (\ref{eq:ipasstau}), we have
\begin{equation}
\lim_{i\rightarrow+\infty}\int_{\partial^{+}B_{\frac{R}{\delta_{i}}}^{+}}\left[\frac{\partial j_{n}}{\partial\eta_{i}}w_{i}-\frac{\partial w_{i}}{\partial\eta_{i}}j_{n}\right]d\sigma_{\hat{g}_{i}}=0\label{eq:nullo1}
\end{equation}
and by (\ref{eq:wi}) and by the decay of $Q_{i}$ given in (\ref{eq:Q-1-1})
we have 
\begin{equation}
\lim_{i\rightarrow+\infty}\int_{B_{\frac{R}{\delta_{i}}}^{+}}j_{n}L_{\hat{g}_{i}}w_{i}d\mu_{\hat{g}_{i}}=\int_{B_{\frac{R}{\delta_{i}}}^{+}}j_{n}Q_{i}d\mu_{\hat{g}_{i}}=0.\label{eq:nullo2}
\end{equation}
Finally, since $\Delta j_{n}=0$, by (\ref{eq:L-Delta}) we get 
\begin{equation}
\lim_{i\rightarrow+\infty}\int_{B_{\frac{R}{\delta_{i}}}^{+}}w_{i}L_{\hat{g}_{i}}j_{n}d\mu_{\hat{g}_{i}}=0,\label{eq:nullo3}
\end{equation}
thus by (\ref{eq:nullo1}) (\ref{eq:nullo2}) and (\ref{eq:nullo3})
we have
\begin{align}
\lim_{i\rightarrow+\infty}\int_{\partial'B_{\frac{R}{\delta_{i}}}^{+}}j_{n}F_{i}d\sigma_{\hat{g}_{i}} & =\lim_{i\rightarrow+\infty}\int_{\partial'B_{\frac{R}{\delta_{i}}}^{+}}w_{i}\left[B_{\hat{g}_{i}}j_{n}+b_{i}j_{n}\right]d\sigma_{\hat{g}_{i}}\nonumber \\
 & =\int_{\partial\mathbb{R}_{+}^{n}}w\left[\frac{\partial j_{n}}{\partial y_{n}}+nU^{\frac{2}{n-2}}j_{n}\right]d\sigma_{\hat{g}_{i}}=0\label{eq:nullofinale}
\end{align}
since $\frac{\partial j_{n}}{\partial y_{n}}+nU^{\frac{2}{n-2}}j_{n}=0$
when $y_{n}=0$. Comparing (\ref{eq:neg}) and (\ref{eq:nullofinale})
we get the contradiction.
\end{proof}
The above lemmas are the core of the following proposition, in which
we iterate the procedure of Lemma \ref{lem:coreLemma}, to obtain
better estimates of the rescaled solution $v_{i}$ of (\ref{eq:Prob-hat})
around the isolated simple blow up point $x_{i}\rightarrow x_{0}$.
\begin{prop}
\label{prop:stimawi}Assume $n\ge8$. Let $\gamma_{x_{i}}$ be defined
in (\ref{eq:vqdef}). There exist $R,C>0$ such that 
\begin{align*}
|\nabla_{\bar{y}}^{\tau}v_{i}(y)-U(y)-\delta_{i}^{2}\gamma_{x_{i}}(y)| & \le C\delta_{i}^{3}(1+|y|)^{5-\tau-n}\\
\left|y_{n}\frac{\partial}{\partial_{n}}\left(v_{i}(y)-U(y)-\delta_{i}^{2}\gamma_{x_{i}}(y)\right)\right| & \le C\delta_{i}^{3}(1+|y|)^{5-n}
\end{align*}
for $|y|\le\frac{R}{2\delta_{i}}$. Here $\tau=0,1,2$ and $\nabla_{\bar{y}}^{\tau}$
is the differential operator of order $\tau$ with respect the first
$n-1$ variables. 
\end{prop}
\begin{proof}
In analogy with Lemma \ref{lem:coreLemma}, we set 
\[
w_{i}(y):=v_{i}(y)-U(y)-\delta_{i}^{2}\gamma_{x_{i}}(y)\text{ for }|y|\le R/\delta_{i},
\]
and we have that $w_{i}$ satisfies (\ref{eq:wi}) where 
\begin{align*}
b_{i}= & (n-2)\frac{v_{i}^{\frac{n}{n-2}}-(U+\delta_{i}^{2}\gamma_{x_{i}})^{\frac{n}{n-2}}}{v_{i}-U-\delta_{i}^{2}\gamma_{x_{i}}}\\
\bar{Q}_{i}= & -\frac{1}{\varepsilon_{i}\delta_{i}}\left\{ (n-2)(U+\delta_{i}^{2}\gamma_{x_{i}})^{\frac{n}{n-2}}-(n-2)U^{\frac{n}{n-2}}-n\delta_{i}^{2}U^{\frac{2}{n-2}}\gamma_{x_{i}}-\frac{n-2}{2}h_{\hat{g}_{i}}(U+\delta_{i}^{2}\gamma_{x_{i}})\right\} \\
F_{i}= & \bar{Q}_{i}+\varepsilon_{i}\delta_{i}\alpha_{i}(\delta_{i}y)v_{i}(y)\\
Q_{i}= & -\frac{1}{\varepsilon_{i}\delta_{i}}\left\{ \left(L_{\hat{g}_{i}}-\Delta\right)(U+\delta_{i}^{2}\gamma_{x_{i}})+\delta_{i}^{2}\Delta\gamma_{x_{i}}\right\} .
\end{align*}
As before, $b_{i}$ satisfies inequality (\ref{eq:b2}) and 
\begin{align}
Q_{i} & =O(\delta_{i}^{3}\left(1+|y|\right)^{3-n})\label{eq:Q-1}\\
F_{i} & =O(\delta_{i}^{4}\left(1+|y|\right)^{5-n})+O(\delta_{i}^{3}\left(1+|y|\right)^{2-n})\label{eq:Q-2}
\end{align}
We define again the Green function $G_{i}$ as in the previous lemma
and we have, by Green formula,
\begin{align}
|w_{i}(y)|\le & \int_{B_{\frac{R}{\delta_{i}}}^{+}}|\xi-y|^{2-n}Q_{i}(\xi)d\xi+\int_{\partial^{+}B_{\frac{R}{\delta_{i}}}^{+}}|\xi-y|^{1-n}w_{i}(\xi)d\sigma(\xi)\nonumber \\
 & +\int_{\partial'B_{\frac{R}{\delta_{i}}}^{+}}|\bar{\xi}-y|^{2-n}b_{i}(\xi)w_{i}(\xi)d\bar{\xi})+\int_{\partial'B_{\frac{R}{\delta_{i}}}^{+}}|\bar{\xi}-y|^{2-n}\bar{F}_{i}^{*}(\xi)d\bar{\xi}.\label{eq:Green}
\end{align}
By the results of Lemma \ref{lem:coreLemma} and Lemma \ref{lem:taui},
and in analogy with equation (\ref{eq:wibordo}) we have that 
\begin{eqnarray}
|w_{i}(y)|\le C\delta_{i}^{3}\text{ on }B_{R/\delta_{i}}^{+} & \text{ and } & |w_{i}(\xi)|\le C\delta_{i}^{n-2}\text{ on }\partial^{+}B_{R/\delta_{i}}^{+}.\label{eq:w_i-improved}
\end{eqnarray}
Plugging (\ref{eq:b2}), (\ref{eq:Q-1}), (\ref{eq:Q-2}) and (\ref{eq:w_i-improved})
in (\ref{eq:Green}) and proceeding as in Lemma \ref{lem:coreLemma}
we obtain
\begin{align}
\int_{B_{\frac{R}{\delta_{i}}}^{+}}|\xi-y|^{2-n}Q_{i}(\xi)d\xi & \le C\delta_{i}^{3}(1+|y|)^{5-n}\label{eq:W1-1}\\
\int_{\partial^{+}B_{\frac{R}{\delta_{i}}}^{+}}|\xi-y|^{1-n}w_{i}(\xi)d\sigma(\xi) & \le C\delta_{i}^{n-2}\label{eq:W2-1}\\
\int_{\partial'B_{\frac{R}{\delta_{i}}}^{+}}|\bar{\xi}-y|^{2-n}b_{i}(\xi)w_{i}(\xi)d\bar{\xi}) & \le\delta_{i}^{3}(1+|y|)^{-1}\label{eq:W3-1}\\
\int_{\partial'B_{\frac{R}{\delta_{i}}}^{+}}|\bar{\xi}-y|^{2-n}\bar{Q}_{i}(\xi)d\bar{\xi} & \le C\delta_{i}^{3}(1+|y|)^{5-n}\label{eq:W4-1}\\
\int_{\partial'B_{\frac{R}{\delta_{i}}}^{+}}|\bar{\xi}-y|^{2-n}\varepsilon_{i}\delta_{i}\alpha_{i}(\delta_{i}\xi)v_{i}(\xi)d\bar{\xi} & \le C\delta_{i}^{3}(1+|y|)^{5-n}\label{eq:W5-1}
\end{align}
so
\begin{equation}
|w_{i}(y)|\le C\delta_{i}^{3}(1+|y|)^{-1}\text{ for }|y|\le\frac{R}{2\delta_{i}}.\label{eq:wi-improv2}
\end{equation}
As before, we iterate the procedure until we reach 
\begin{equation}
|w_{i}(y)|\le C\delta_{i}^{3}(1+|y|)^{5-n}\text{ for }|y|\le\frac{R}{2\delta_{i}},\label{eq:wi-improv2-1-1}
\end{equation}
which proves the first claim for $\tau=0$. The other claims follow
as in the previous proofs.
\end{proof}

\section{Sign estimates of Pohozaev identity terms\label{sec:Sign-estimates}}

In this section, we want to estimate $P(u_{i},r)$, where $\left\{ u_{i}\right\} _{i}$
is a family of solutions of (\ref{eq:Prob-i}) which has an isolated
simple blow up point $x_{i}\rightarrow x_{0}$. This estimate, given
in the following Proposition \ref{prop:segno}, is a crucial point
for the proof of the vanishing of the Weyl tensor at an isolated simple
blow up point.

Since the leading term of $P(u_{i},r)$ will be $-\int_{B_{r/\delta_{i}}^{+}}\left(y^{b}\partial_{b}u+\frac{n-2}{2}u\right)\left[(L_{\hat{g}_{i}}-\Delta)v\right]dy$
we set
\begin{equation}
R(u,v)=-\int_{B_{r/\delta_{i}}^{+}}\left(y^{b}\partial_{b}u+\frac{n-2}{2}u\right)\left[(L_{\hat{g}_{i}}-\Delta)v\right]dy.\label{eq:Ruv}
\end{equation}

\begin{prop}
\label{prop:segno}Let $x_{i}\rightarrow x_{0}$ be an isolated simple
blow-up point for $u_{i}$ solutions of (\ref{eq:Prob-i}). Then,
fixed $r$, we have, for $i$ large 
\begin{align*}
\hat{P}(u_{i},r)\ge & \delta_{i}^{4}\frac{(n-2)\omega_{n-2}I_{n}^{n}}{(n-1)(n-3)(n-5)(n-6)}\left[\frac{\left(n-2\right)}{6}|\bar{W}(x_{i})|^{2}+\frac{4(n-8)}{(n-4)}R_{nlnj}^{2}(x_{i})\right]\\
 & -2\delta_{i}^{4}\int_{\mathbb{R}_{+}^{n}}\gamma_{x_{i}}\Delta\gamma_{x_{i}}dy+o(\delta_{i}^{4}).
\end{align*}
\end{prop}
\begin{proof}
We recall that
\begin{multline*}
\hat{P}(u_{i},r):=-\int\limits _{B_{r}^{+}}\left(y^{a}\partial_{a}u_{i}+\frac{n-2}{2}u_{i}\right)[(L_{g_{i}}-\Delta)u_{i}]dy+\frac{n-2}{2}\int\limits _{\partial'B_{r}^{+}}\left(\bar{y}^{k}\partial_{k}u_{i}+\frac{n-2}{2}u_{i}\right)h_{g_{i}}u_{i}d\bar{y}\\
+\frac{n-2}{2}\int\limits _{\partial'B_{r}^{+}}\left(\bar{y}^{k}\partial_{k}u_{i}+\frac{n-2}{2}u_{i}\right)\varepsilon_{i}\alpha_{i}u_{i}d\bar{y}.
\end{multline*}
where $B_{r}^{+}$ is the counter-image of $B_{g_{i}}^{+}(x_{i},r)$
by $\psi_{x_{i}}$. Now, set 
\[
v_{i}(y):=\delta_{i}^{\frac{n-2}{2}}u_{i}(\delta_{i}y)\text{ for }y\in B_{\frac{R}{\delta_{i}}}^{+}(0)
\]
After a change of variables we have 
\[
\int\limits _{\partial'B_{r}^{+}}\left(\bar{y}^{k}\partial_{k}u_{i}+\frac{n-2}{2}u_{i}\right)\varepsilon_{i}\alpha_{i}u_{i}d\bar{y}=\varepsilon_{i}\delta_{i}\int\limits _{\partial'B_{r/\delta_{i}}^{+}}\left(\bar{y}^{k}\partial_{k}v_{i}+\frac{n-2}{2}v_{i}\right)\alpha_{i}(\delta_{i}y)v_{i}d\bar{y}.
\]
By Proposition \ref{prop:stimawi} and by (\ref{eq:gradvq}) of Lemma
\ref{lem:vq}, for $|y|<R/\delta_{i}$ we have
\[
\left|v_{i}(y)-U(y)\right|=O(\delta_{i}^{3}(1+|y|^{5-n})+O(\delta_{i}^{2}(1+|y|^{4-n})=O(\delta_{i}^{2}(1+|y|^{4-n})
\]
\[
\left|y_{k}\partial_{k}v_{i}(y)-y_{k}\partial_{k}U(y)\right|=O(\delta_{i}^{3}(1+|y|^{5-n})+O(\delta_{i}^{2}(1+|y|^{4-n})=O(\delta_{i}^{2}(1+|y|^{4-n}),
\]
so 
\[
\int\limits _{\partial'B_{r}^{+}}\left(\bar{y}^{k}\partial_{k}u_{i}+\frac{n-2}{2}u_{i}\right)\varepsilon_{i}\alpha_{i}u_{i}d\bar{y}=\varepsilon_{i}\delta_{i}\int\limits _{\partial'B_{r/\delta_{i}}^{+}}\left(\bar{y}^{k}\partial_{k}U+\frac{n-2}{2}U\right)\alpha_{i}(\delta_{i}y)Ud\bar{y}+\varepsilon_{i}\delta_{i}O(\delta_{i}^{2})
\]
and, recalling that $\alpha_{i}(\delta_{i}y)\rightarrow\alpha(x_{0})<0$
and proceeding as in (\ref{eq:avj}) we get
\[
\lim_{i\rightarrow\infty}\int\limits _{\partial'B_{r/\delta_{i}}^{+}}\left(\bar{y}^{k}\partial_{k}U+\frac{n-2}{2}U\right)\alpha_{i}(\delta_{i}y)Ud\bar{y}=\frac{n-2}{2}\alpha(x_{0})\int\limits _{\mathbb{R}^{n-1}}\frac{1-|\bar{y}|^{2}}{\left[1+|\bar{y}|^{2}\right]^{n-1}}d\bar{y}>0.
\]

Thus, for $i$ sufficiently large we obtain 
\begin{align*}
\hat{P}(u_{i},r) & \ge-\int_{B_{r/\delta_{i}}^{+}}\left(y^{b}\partial_{b}y+\frac{n-2}{2}v_{i}\right)\left[(L_{\hat{g}_{i}}-\Delta)v_{i}\right]dy\\
 & +\frac{n-2}{2}\int_{\partial'B_{r/\delta_{i}}^{+}}\left(y^{b}\partial_{b}v_{i}+\frac{n-2}{2}v_{i}\right)h_{g_{i}}(\delta_{i}y)v_{i}d\bar{y}.
\end{align*}
Since $h_{g_{i}}(\delta_{i}y)=O(\delta_{i}^{4}|y|^{4})$ we have

\begin{multline*}
\int_{\partial'B_{r}^{+}}\left(y^{b}\partial_{b}v_{i}+\frac{n-2}{2}v_{i}\right)h_{g_{i}}(\delta_{i}y)v_{i}d\bar{y}\\
=O(\delta_{i}^{5})\int_{\partial'B_{r}^{+}}(1+|y|)^{4-2n}|y|^{4}dy=O(\delta_{i}^{5})\text{ for }n\ge8.
\end{multline*}
So 
\[
\hat{P}(u_{i},r)\ge-\int_{B_{r/\delta_{i}}^{+}}\left(y^{b}\partial_{b}v_{i}+\frac{n-2}{2}v_{i}\right)\left[(L_{\hat{g}_{i}}-\Delta)v_{i}\right]dy+O(\delta_{i}^{5})
\]
for $i$ sufficiently large. Now define, in analogy with Proposition
\ref{prop:stimawi}, 
\[
w_{i}(y):=v_{i}(y)-U(y)-\delta_{i}^{2}\gamma_{x_{i}}(y).
\]
Recalling (\ref{eq:Ruv}), we have 
\begin{align*}
\hat{P}(u_{i},r) & \ge R(U,U)+R(U,\delta_{i}^{2}\gamma_{x_{i}})+R(\delta_{i}^{2}\gamma_{x_{i}},U)+R(w_{i},U)+R(U,w_{i})\\
 & +R(w_{i,}w_{i})+R(\delta_{i}^{2}\gamma_{q},\delta_{i}^{2}\gamma_{x_{i}})+R(w_{i},\delta_{i}^{2}\gamma_{x_{i}})+R(\delta_{i}^{2}\gamma_{x_{i}},w_{i})+O(\delta_{i}^{5})
\end{align*}
and, by the following Lemma \ref{lem:R(U,U)} we conclude 
\begin{align*}
\hat{P}(u_{i},r)\ge & R(U,U)+R(U,\delta_{i}^{2}\gamma_{x_{i}})+R(\delta_{i}^{2}\gamma_{x_{i}},U)+o(\delta_{i}^{4})\\
= & \delta_{i}^{4}\frac{(n-2)\omega_{n-2}I_{n}^{n}}{(n-1)(n-3)(n-5)(n-6)}\left[\frac{\left(n-2\right)}{6}|\bar{W}(x_{i})|^{2}+\frac{4(n-8)}{(n-4)}R_{nlnj}^{2}(x_{i})\right]\\
 & -2\delta_{i}^{4}\int_{\mathbb{R}_{+}^{n}}\gamma_{x_{i}}\Delta\gamma_{x_{i}}dy+o(\delta_{i}^{4})
\end{align*}
and we prove the result.
\end{proof}
\begin{lem}
\label{lem:R(U,U)}For $n\ge8$ we have 
\begin{align*}
 & R(U,U)=\delta^{4}\frac{(n-2)\omega_{n-2}I_{n}^{n}}{(n-1)(n-3)(n-5)(n-6)}\left[\frac{\left(n-2\right)}{6}|\bar{W}(q)|^{2}+\frac{4(n-8)}{(n-4)}R_{ninj}^{2}\right]+o(\delta^{4})\\
 & R(U,\delta^{2}\gamma_{q})+R(\delta^{2}\gamma_{q},U)=-2\delta^{4}\int_{\mathbb{R}_{+}^{n}}\gamma_{q}\Delta\gamma_{q}dy+o(\delta^{4})\\
 & R(\delta^{2}\gamma_{q},\delta^{2}\gamma_{q})=O(\delta^{6})\\
 & R(w_{i},w_{i})=O(\delta^{6})\\
 & R(U,w_{i})+R(w_{i},U)=O(\delta^{5})\\
 & R(\delta^{2}\gamma_{q},w_{i})+R(w_{i},\delta^{2}\gamma_{q})=O(\delta^{5})
\end{align*}
\end{lem}
\begin{proof}
For the proof we refer to \cite{GM}.
\end{proof}
\begin{prop}
\label{prop:7.1}Let $x_{i}\rightarrow x_{0}$ be an isolated simple
blow-up point for $u_{i}$ solutions of (\ref{eq:Prob-i}). Then
\begin{enumerate}
\item If $n=8$ then $|\bar{W}(x_{0})|=0.$
\item If $n>8$ then $|W(x_{0})|=0.$
\end{enumerate}
\end{prop}
\begin{proof}
By Proposition \ref{prop:4.3} and Proposition \ref{prop:Lemma 4.4},
and since $M_{i}=\delta_{i}^{\frac{2-n}{2}}$ we have,
\begin{align*}
P(u_{i},r):= & \frac{1}{M_{i}^{2\lambda_{i}}}\int\limits _{\partial^{+}B_{r}^{+}}\left(\frac{n-2}{2}M_{i}^{\lambda_{i}}u_{i}\frac{\partial M_{i}^{\lambda_{i}}u_{i}}{\partial r}-\frac{r}{2}|\nabla M_{i}^{\lambda_{i}}u_{i}|^{2}+r\left|\frac{\partial M_{i}^{\lambda_{i}}u_{i}}{\partial r}\right|^{2}\right)d\sigma_{r}\\
 & +\frac{r(n-2)^{2}}{\left(n-1\right)M_{i}^{\lambda_{i}\frac{2(n-1)}{n-2}}}\int\limits _{\partial(\partial'B_{r}^{+})}\left(M_{i}^{\lambda_{i}}u_{i}\right)^{\frac{2(n-1)}{n-2}}d\bar{\sigma}_{g}.\\
\le & \frac{C}{M_{i}^{\lambda_{i}\frac{2(n-1)}{n-2}}}\le C\delta_{i}^{(n-1)\lambda_{i}}\le C\delta_{i}^{n-2}.
\end{align*}
On the other hand recalling Proposition \ref{prop:segno} and Theorem
\ref{thm:poho} we have 

\[
P(u_{i},r)=\hat{P}(u_{i},r)\ge\delta_{i}^{4}\frac{(n-2)\omega_{n-2}I_{n}^{n}}{(n-1)(n-3)(n-5)(n-6)}\left[\frac{\left(n-2\right)}{6}|\bar{W}(x_{i})|^{2}+\frac{4(n-8)}{(n-4)}R_{nlnj}^{2}(x_{i})\right]+o(\delta_{i}^{4}),
\]
because $\int\gamma_{x_{i}}\Delta\gamma_{x_{i}}\le0$ (see (\ref{new})
of Lemma \ref{lem:vq}) so we get $|\bar{W}(x_{i})|\le\delta_{i}^{2}$
if $n=8$, and $\left[\frac{\left(n-2\right)}{6}|\bar{W}(x_{i})|^{2}+\frac{4(n-8)}{(n-4)}R_{nlnj}^{2}(x_{i})\right]\le\delta_{i}^{2}$
if $n>8$. For the case $n>8$ we recall that when the boundary is
umbilic $W(q)=0$ if and only if $\bar{W}(q)=0$ and $R_{nlnj}(q)=0$
(see \cite[page 1618]{M1}), and we conclude the proof. 
\end{proof}
\begin{rem}
\label{rem:P'}Let $x_{i}\rightarrow x_{0}$ be an isolated blow up
point for $u_{i}$ solutions of (\ref{eq:Prob-i}). We set 
\begin{equation}
P'\left(u,r\right):=\int\limits _{\partial^{+}B_{r}^{+}}\left(\frac{n-2}{2}u\frac{\partial u}{\partial r}-\frac{r}{2}|\nabla u|^{2}+r\left|\frac{\partial u}{\partial r}\right|^{2}\right)d\sigma_{r},\label{eq:P'def}
\end{equation}
 so
\[
P(u_{i},r)=P'(u_{i},r)+\frac{r(n-2)^{2}}{(n-1)}\int\limits _{\partial(\partial'B_{r}^{+})}u_{i}^{\frac{2(n-1)}{n-2}}d\bar{\sigma}_{g}
\]
and, keeping in mind that for $i$ large $M_{i}u_{i}\le C|y|^{2-n}$
by Proposition \ref{prop:4.3}, we have 
\begin{equation}
\left|r\int\limits _{\partial(\partial'B_{r}^{+})}u_{i}^{\frac{2(n-1)}{n-2}}d\bar{\sigma}_{g}\right|\le\frac{Cr}{M_{i}^{\frac{2(n-1)}{n-2}}}\int_{\begin{array}{c}
y_{n}=0\\
|\bar{y}|=r
\end{array}}\frac{1}{|y|^{2(n-1)}}d\bar{\sigma}_{g}\le\frac{C(r)}{M_{i}^{\frac{2(n-1)}{n-2}}}=C(r)\delta_{i}^{n-2}\label{eq:5.15-1}
\end{equation}
for $i$ sufficiently large.

Using Proposition \ref{prop:segno}, (\ref{eq:5.15-1}), and since
$n\ge8$ we get 
\begin{equation}
P'(u_{i},r)=P(u_{i},r)-\frac{r(n-2)^{2}}{(n-1)}\int\limits _{\partial(\partial'B_{r}^{+})}u_{i}^{\frac{2(n-1)}{n-2}}d\bar{\sigma}_{g}\ge A\delta_{i}^{4}+o(\delta^{4})\label{eq:stimaP'}
\end{equation}
where $A>0$.
\end{rem}
\begin{prop}
\label{prop:isolato->semplice}Let $x_{i}\rightarrow x_{0}$ be an
isolated blow up point for $u_{i}$ solutions of (\ref{eq:Prob-i}).
Assume $n=8$ and $|\bar{W}(x_{0})|\neq0$ or $n>8$ and $|W(x_{0})|\neq0$.
Then $x_{0}$ is isolated simple. 
\end{prop}
For the proof of this Lemma we refer to \cite{Al,GM}

\section{A splitting lemma\label{sec:A-splitting-lemma}}

The first result in this section are analogous to \cite[Proposition 5.1]{LZ},
\cite[Lemma 3.1]{SZ}, \cite[Proposition 1.1]{HL} and \cite[Proposition 4.2]{Al},
so the proof will be omitted.
\begin{prop}
\label{prop:4.2}Given $\beta>0$ and $R>0$ there exist two constants
$C_{0},C_{1}>0$ (depending on $\beta$, $R$ and $(M,g)$) such that
if $u$ is a solution of 
\begin{equation}
\left\{ \begin{array}{cc}
L_{g}u=0 & \text{ in }M\\
\frac{\partial u}{\partial\nu}+\frac{n-2}{2}h_{g}u+\varepsilon\alpha u=(n-2)u^{\frac{n}{n-2}} & \text{ on }\partial M
\end{array}\right.\label{eq:Prob-p}
\end{equation}
and $\max_{\partial M}u>C_{0}$, then $\tau:=\frac{n}{n-2}-p<\beta$
and there exist $q_{1},\dots,q_{N}\in\partial M$, with $N=N(u)\ge1$
with the following properties: for $j=1,\dots,N$ 
\begin{enumerate}
\item set $r_{j}:=Ru(q_{j})^{1-p}$ then $\left\{ B_{r_{j}}\cap\partial M\right\} _{j}$
are a disjoint collection;
\item we have $\left|u(q_{j})^{-1}u(\psi_{j}(y))-U(u(q_{j})^{p-1}y)\right|_{C^{2}(B_{2r_{j}}^{+})}<\beta$
(here $\psi_{j}$ are the Fermi coordinates at point $q_{j}$;
\item we have
\begin{align}
u(x)d_{\bar{g}}\left(x,\left\{ q_{1},\dots,q_{n}\right\} \right)^{\frac{1}{p-1}}\le C_{1} & \text{ for all }x\in\partial M\label{eq:Claim3-1}\\
u(q_{j})d_{\bar{g}}\left(q_{j},q_{k}\right)^{\frac{1}{p-1}}\ge C_{0} & \text{ for any }j\neq k.\label{eq:Claim3-2}
\end{align}
Here $\bar{g}$ is the geodesic distance on $\partial M$.
\end{enumerate}
\end{prop}
Now we prove that only isolated blow up points may occur to a blowing
up sequence of solution. For the proof of the next proposition we
refer to \cite{GM}
\begin{prop}
\label{prop:splitting}Assume $n\ge8$. Given $\beta,R>0$, consider
$C_{0},C_{1}$ as in the previous proposition. Assume $W(x)\neq0$
for any $x\in\partial M$ if $n>8$ or $\bar{W}(x)\neq0$ for any
$x\in\partial M$ if $n=8$. Then there exists $d=d(\beta,R)$ such
that for any $u$ solution of (\ref{eq:Prob-p}) with $\max_{\partial M}u>C_{0}$,
we have 
\[
\min_{\begin{array}{c}
i\neq j\\
1\le i,j\le N(u)
\end{array}}d_{\bar{g}}(q_{i}(u),q_{j}(u))\ge d,
\]
where $q_{1}(u),\dots q_{N}(u)$ and $N=N(u)$ are given in the previous
proposition. 
\end{prop}

\section{Proof of the main result\label{sec:Main-Proof}}
\begin{proof}[Proof of Theorem \ref{thm:main}]
. By contradiction, suppose that $x_{i}\rightarrow x_{0}$ is a blowup
point for $u_{i}$ solutions of (\ref{eq:Prob-2}). Let $q_{1}^{i},\dots q_{N(u_{i})}^{i}$
the sequence of points given by Proposition \ref{prop:4.2}. By Claim
3 of Proposition \ref{prop:4.2} there exists a sequence of indices
$k_{i}\in1,\dots N$ such that $d_{\bar{g}}\left(x_{i},q_{k_{i}}^{i}\right)\rightarrow0$.
Up to relabeling, we say $k_{i}=1$ for all $i$. Then also $q_{1}^{i}\rightarrow x_{0}$
is a blow up point for $u_{i}$. By Proposition \ref{prop:splitting}
and Proposition \ref{prop:isolato->semplice} we have that $q_{1}^{i}\rightarrow x_{0}$
is an isolated simple blow up point for $u_{i}$. Then by Proposition
\ref{prop:7.1} we deduce that $\bar{W}(x_{0})=0$ if $n=8$ or that
$W(x_{0})=0$ if $n>8$, which contradicts the assumption of this
theorem and proves the result. 
\end{proof}

\section{Proof of Theorem \ref{thm:main-1}\label{sec:almaraz}}

In this case the manifold is not umbilic, so, we have a different
expansion of the metric. Firstly, there exists a metric $\tilde{g}$,
conformal to $g$, such that $h_{\tilde{g}}\equiv0$ (see \cite[Prop. 3.1]{M1}).
So, we can suppose w.l.o.g. that $h_{g}\equiv0$ in the original problem,
that is 
\[
\left\{ \begin{array}{cc}
L_{g}u=0 & \text{ in }M\\
\frac{\partial u}{\partial\nu}+\varepsilon\alpha u=(n-2)u^{\frac{n}{n-2}} & \text{ on }\partial M
\end{array}\right.
\]
This leads to obvious modification in the Pohozaev identity. The expansion
of the metric in this case is 
\begin{align}
|g(y)|^{1/2}= & 1-\frac{1}{2}\left[\|\pi\|^{2}+\text{Ric}(0)\right]y_{n}^{2}-\frac{1}{6}\bar{R}_{ij}(0)y_{i}y_{j}+O(|y|^{3})\label{eq:|g|-1}\\
g^{ij}(y)= & \delta_{ij}+2h_{ij}(0)y_{n}+\frac{1}{3}\bar{R}_{ikjl}(0)y_{k}y_{l}+2\frac{\partial h_{ij}}{\partial y_{k}}(0)ty_{k}\nonumber \\
 & +\left[R_{injn}(0)+3h_{ik}(0)h_{kj}(0)\right]y_{n}^{2}+O(|y|^{3})\label{eq:gij-1}\\
g^{an}(y)= & \delta_{an}\label{eq:gin}
\end{align}
where $\pi$ is the second fundamental form and $h_{ij}(0)$ are its
coefficients, and $\text{Ric}(0)=R_{nini}(0)=R_{nn}(0)$ (see \cite{Es}).

The main difference with the previous case lies in the second order
approximation of the solution near an isolated simple blow up point.
We define here, as in \cite[Section 5]{Al} $\hat{\gamma}_{q}:\mathbb{R}_{+}^{n}\rightarrow\mathbb{R}$
is the unique solution of the problem 
\begin{equation}
\left\{ \begin{array}{ccc}
-\Delta\gamma=2h_{ij}(q)t\partial_{ij}^{2}U &  & \text{on }\mathbb{R}_{+}^{n};\\
\frac{\partial\gamma}{\partial t}+nU^{\frac{2}{n-2}}\gamma=0 &  & \text{on \ensuremath{\partial}}\mathbb{R}_{+}^{n}.
\end{array}\right.\label{eq:vqdef-1}
\end{equation}
such that $\hat{\gamma}_{q}$ is $L^{2}(\mathbb{R}_{+}^{n})$-orthogonal
to $j_{b}$ for all $b=1,\dots,n$. Again, we have that (see \cite[Section 5]{Al}
and \cite[Section 2]{GMP18} for the proofs).
\begin{equation}
|\nabla^{r}v_{q}(y)|\le C(1+|y|)^{3-r-n}\text{ for }r=0,1,2,\label{eq:gradvq-1}
\end{equation}
\begin{equation}
\int_{\partial\mathbb{R}_{+}^{n}}U^{\frac{n}{n-2}}v_{q}=0\label{eq:Uvq-1}
\end{equation}
\begin{equation}
\int_{\partial\mathbb{R}_{+}^{n}}\Delta v_{q}v_{q}dzdt\le0,\label{new-1}
\end{equation}

In this case we will have the following result (see \cite[Proposition 6.1]{Al})
which replaces Proposition \ref{prop:stimawi}
\begin{prop}
\label{prop:stimawi-1}Assume $n\ge7$. Let $\hat{\gamma}_{x_{i}}$
be defined in (\ref{eq:vqdef-1}). There exist $R,C>0$ such that
\begin{align*}
|\nabla_{\bar{y}}^{\tau}v_{i}(y)-U(y)-\delta_{i}\hat{\gamma}_{x_{i}}(y)| & \le C\delta_{i}^{2}(1+|y|)^{4-\tau-n}\\
\left|y_{n}\frac{\partial}{\partial_{n}}\left(v_{i}(y)-U(y)-\delta_{i}^{2}\hat{\gamma}_{x_{i}}(y)\right)\right| & \le C\delta_{i}^{2}(1+|y|)^{4-n}
\end{align*}
for $|y|\le\frac{R}{2\delta_{i}}$. 
\end{prop}
By the expansion of the metric, the Pohozaev identity and Proposition
\ref{prop:stimawi-1} we have the following estimate on the sign condition
which corresponds to Proposition \ref{prop:segno} 
\begin{prop}
\label{prop:segno-1}Let $x_{i}\rightarrow x_{0}$ be an isolated
simple blow-up point for $u_{i}$ solutions of (\ref{eq:Prob-i}).
Then, fixed $r$, we have, for $i$ large 
\begin{align*}
P(u_{i},r)\ge & \delta_{i}^{2}\frac{(n-6)\omega_{n-2}I_{n}^{n}}{(n-1)(n-2)(n-3)(n-4)}\left[|h_{kl}(x_{i})|^{2}\right]+o(\delta_{i}^{2})
\end{align*}
\end{prop}
\begin{proof}
As in Proposition \ref{prop:segno}, we use that $\alpha<0$ to get
that 
\[
P(u_{i},r)\ge-\int_{B_{r/\delta_{i}}^{+}}\left(y^{b}\partial_{b}y+\frac{n-2}{2}v_{i}\right)\left[(L_{\hat{g}_{i}}-\Delta)v_{i}\right]dy.
\]
Then, by the estimates contained in \cite[Theorem 7.1]{Al}, and in
light of (\ref{new-1}) we get the proof.
\end{proof}
At this point we have all the tools to prove Theorem \ref{thm:main-1}
using the same strategy of Section \ref{sec:Main-Proof}.

\end{document}